\documentclass[11pt,a4paper]{article}%
\usepackage[centertags]{amsmath}
\usepackage{amsfonts}
\usepackage{amssymb}
\usepackage{amsthm}
\usepackage{epsfig}
\usepackage{setspace}
\usepackage{ae}
\usepackage{eucal}
\usepackage[usenames]{color}%
\usepackage{graphicx}

\theoremstyle{plain}
\newtheorem{thm}{Theorem}[section]

\newtheorem{cor}[thm]{Corollary}
\newtheorem{prop}[thm]{Proposition}
\theoremstyle{definition}

\theoremstyle{remark}

\newtheorem{rem}[thm]{Remark}

\renewcommand{\div}{\operatorname{div}}

\begin{document}

\title{A global scheme for the incompressible Navier-Stokes equation on compact Riemannian manifolds}
\author{J\"org Kampen }
\maketitle

\begin{abstract}
 We propose a global scheme for a controlled Navier-Stokes equation system on compact smooth Riemannian manifolds. Global upper bounds for the controlled velocity function and for the control function itself are constructed which leads to global existence of solutions.  
\end{abstract}

\section{Introduction}
Let $M$ be a compact smooth Riemannian manifold of dimension $n$ with line element
\begin{equation}
ds^2=g_{ij}dx^idx^j.
\end{equation}
In order to investigate the Navier-Stokes equation on manifolds the Navier-Stokes equation system on flat manifolds 
\begin{equation}\label{NavierStokes}
\left\lbrace \begin{array}{ll}
\frac{\partial\mathbf{v}}{\partial t}-\nu \Delta \mathbf{v}+ (\mathbf{v} \cdot \nabla) \mathbf{v} = - \nabla p + \mathbf{f}_{ex},\\
\\
\nabla \cdot \mathbf{v} = 0,\\
\\
\mathbf{v}(0,.)=\mathbf{h},
\end{array}\right.
\end{equation}
has to be reinterpreted. The class of flat compact manifolds where (\ref{NavierStokes}) is a correct formulation is rather limited (the classification is well-studied).  The $n$-torus $\mathbb{T}^n$ is the most natural example in the sense that this is the only type of flat compact manifolds which occurs in any dimension $n\geq 1$. In \cite{BS} it is mentioned that the Galilei symmetry of the equations on flat spaces fixes the highly constrained structure of the equation, especially the coefficient of the nonlinear convection term. As we mentioned in \cite{KHNS}, although there are rather natural generalisations of the Navier Stokes equation model on Riemannian manifolds, there is some freedom of choice concerning the description of the coupling of the velocity field to the curvature in such cases. Locally, and in a suitable chart on $M$ the equation on the manifold $M$ may look as in (\ref{NavierStokes}), but it seems that no natural law is known which determines how such local equation systems have to be 'glued' in order to get a global equation, although there may be 'natural choices' from a mathematical point of view. 
The following considerations are not a derivation of the Navier Stokes equation on manifolds - this is well-known and related considerations may be found in \cite{MT}. The purpose of the following considerations is tautological from this point of view. Their purpose is the preparation of a solution scheme, which will be formulated explicitly in the second section.   
Well, in our context of Riemannian manifolds in this paper it is natural to replace the nonlinear term by the covariant derivative of the vector field with respect to itself, i.e. 
\begin{equation}
(\mathbf{v} \cdot \nabla) \mathbf{v}~\mbox{ is replaced naturally by } \nabla_{\mathbf{v}}\mathbf{v},
\end{equation}
where the latter symbol denotes the covariant derivative in standard invariant notation.
That this requirement is consistent with the formulation on flat manifolds can be seen easily
by writing the covariant derivative in coordinates. For the $j$th component we have (we use  Einstein notation in the following if convenient)
\begin{equation}
\left( \nabla_{\mathbf{v}}\mathbf{v}\right)_j=\sum_{k=1}^n v^j_{;k}v^k e_j=\sum_{k=1}^n\left( v^j_{,k}+\sum_{m=1}^nv^m\Gamma^j_{mk}\right)v^ke_j 
\end{equation}
where $e_j$ denotes the $j$th unit vector of the Euclidean basis. The Christoffel symbols $\Gamma^j_{mk}$ become zero if the manifold is flat. Hence taking the covariant derivative is a quite natural extension which collapses to the classical equation term for the $n$-torus. Recall that
\begin{equation}
\Gamma^l_{ij}=\frac{1}{2}g^{kl}
\left( g_{jk,i}+g_{ik,j}+g_{ij,k}\right), 
\end{equation}
where $\left( g^{kl}\right) $ denotes the inverse of $\left(g_{ij}\right)$.
Another matter is the reinterpretation of the Laplacian. We can define it in terms of the covariant derivative and its adjoint or in terms of the exterior derivative and its adjoint. The former possibility leads to the Bochner Laplacian
\begin{equation}
L_B\equiv-\Delta^*\Delta
\end{equation}
where 
\begin{equation}
\Delta^*:C^{\infty}\left(M,T^*\otimes T \right)\rightarrow C^{\infty}\left(M,T \right)  
\end{equation}
denotes the adjoint of the connection $\Delta$  on the tangent bundle on $M$ and $T\equiv TM$ (resp. $T^*\equiv T^*M$) denotes the tangent bundle (resp. cotangent bundle). This $\Delta$ is not to be confused with the usual Laplacian of course, but since we do not use this notation in the following there should be no peril of confusion). Another possibility is called the Hodge Laplacian which is defined by
\begin{equation}
L_H\equiv -\left(d^*d+dd^* \right),
\end{equation}
where $d^*$ denotes the adjoint. Note that $L_H$ is related to a  Dirac operator $D$. Indeed from index theory we know that the difference of $D^2$ and the Bochner-Laplacian (applied to some vector field $\mathbf{v}$) is given in terms of the curvature tensor of the connection on the tangent bundle (applied to $\mathbf{v}$). The incompressibility condition usually leads to a simplification of the interpretation of the Laplacian, and in many cases it seems reasonable to interpret it by the Bochner-Laplacian plus the Ricci-tensor applied to the velocity field. In any case we may subsume a lot of possibilities by assuming that we may substitute the Laplacian by a linear scalar second order diffusion operator $L$ which has local coordinates
\begin{equation}\label{LU}
L_U\mathbf{v}\equiv \sum_{j,k=1}^na_{jk}(x)\frac{\partial^2 v_i}{\partial x_j\partial x_k}+\sum_{k=1}^nb_k(x)\frac{\partial v_i}{\partial x_k} ,
\end{equation}
where the subscript $U$ indicates that the operator is looked at with respect to a chart defined on $U\subset {\mathbb R}^n$. Well, the second order coefficients may also depend on a third index $i$, but let us keep things simple, since this is not essential, because we have no coupling of the second order terms in the Navier-Stokes equation. We may impose uniform ellipticity conditions on the second order terms, postponing possible generalizations involving a H\"{o}rmander condition to subsequent investigations. All these local operators $L_U$ together with an atlas on $M$ define the global operator $L$ on $M$ which may be interpreted as a Hodge-Laplacian or Bochner-Laplacian in specific circumstances. In any case we assume that the operator in (\ref{LU}) is uniform elliptic with bounded smooth coefficients. Finally, we need to reinterpret the incompressibility condition on a manifold and recall the meaning of a gradient on a riemannian manifold. Well, in Einstein notation (and with Einstein summation) this is just
\begin{equation}
\div \mathbf{v}:=v^j_{;j}.
\end{equation}
In local coordinates the divergence may be expressed by 
\begin{equation}
\sum_{i=1}^n\frac{\partial v^i}{\partial x_i}+\sum_{k,i=1}^n v^k\Gamma^i_{ki},
\end{equation}
where in Einstein notation the symbols for sums are suppressed for all indexes which appear 'above' and 'below'. Concerning the gradient of the pressure on a Riemannian manifold we recall that
\begin{equation}
\nabla_Mp=\left(g^{ij}\frac{\partial p}{\partial x_i}\right)_{1\leq j\leq n} 
\end{equation}
The considerations so far lead us to the conclusion that we may define the Cauchy problem for the Navier-Stokes equation on manifolds by the equation system
\begin{equation}\label{NavierStokesman}
\left\lbrace \begin{array}{ll}
\frac{\partial\mathbf{v}}{\partial t}-\nu L \mathbf{v}+ \nabla_{\mathbf{v} } \mathbf{v} = - \nabla_{\tiny M} p ,\\
\\
\div\mathbf{v} = 0,\\
\\
\mathbf{v}(0,.)=\mathbf{h}.
\end{array}\right.
\end{equation}
Locally, we can work this out as in Euclidean space, i.e. we may look at local coordinates in a certain chart with values in $U$ where the Cristoffel symbols disappear and the coefficients $g_{ij}$ of the line element satisfy $g_{ij}=\delta_{ij}$ locally. We have to transform the coefficients $a_{ij}$ and $b_i$ of the operator $L$ accordingly. We call these coefficient functions 'locally flat on $U$' and, keeping our general notation, we just write
\begin{equation}\label{locflat}
a_{jk}^{f,U}, b_k^{f,U}
\end{equation}
in order to indicate that we consider the coefficient functions in a chart where the metric is Euclidean and the affine connections involved collapse to a directional derivative. Then locally on $U$ everything looks as in the Euclidean space, except that globally -as $U$ varies -we have variable coefficients of first and second order, i.e., we have for all $1\leq i\leq n$
\begin{equation}\label{eqiloc}
\begin{array}{ll}
\frac{\partial v^i}{\partial t}-\sum_{j,k=1}^na^{f,U}_{jk}\frac{\partial^2 v^i}{\partial x_j\partial x_k}-\sum_{k=1}^nb^{f,U}_k\frac{\partial v^i}{\partial x_k}\\
\\
+\sum_{k=1}^nv^i_{,k}v^k=-p_{,i}.
\end{array}
\end{equation}
In any natural interpretation of the incompressible Navier Stokes equation on Riemannian manifolds the coefficients in (\ref{locflat}) are locally constant in locally flat coordinates. Obviously, locally flat coordinates are useful in order to eliminate the pressure from the equation. (We shall make use of nonzero Christoffel functions below.)
Derivation for each $i$ of equation (\ref{eqiloc}) with respect to $x_i$ and summing up using incompressibility leads to a Laplacian equation on $U$ in the form
\begin{equation}
-\Delta p=S\left( \mathbf{ v},\nabla \mathbf{v}\right)_U,
\end{equation}
where $S^{f,U}\left( \mathbf{ v},\nabla \mathbf{v}\right)_U:U\rightarrow {\mathbb R}$ is a function defined by 
\begin{equation}\label{sfu}
\begin{array}{ll}
S^{f,U}\left( \mathbf{ v},\nabla \mathbf{v}\right)_U(x):=-\sum_{i=1}^n\sum_{j,k=1}^na^{f,U}_{jk,i}\frac{\partial^2 v^i}{\partial x_j\partial x_k}-\sum_{i=1}^n\sum_{k=1}^nb^{f,U}_{k,i}\frac{\partial v^i}{\partial x_k}\\
\\
+\sum_{i,k=1}^nv^i_{,k}v^k_{,i}.
\end{array}
\end{equation}
Again, the superscripts $f$ and $U$ indicate that we look at the functional $S$ in a locally flat metric on a local domain $U$. Therefore for the first order spatial derivatives of the coefficients we have $a^{f,U}_{jk,i}(x)=0$ and $b^{f,U}_{k,i}(x)=0$ for usual interpretations of the Laplacian on the manifold $M$. In this case the equation (\ref{sfu}) collapses to the usual Leray elimination of the pressure. This is especially true in the case of flat manifolds such as the $n$-torus.  
\begin{rem}
For variable second order terms without coupling we can compute the Leray projection as in the classical model: for each $1\leq j,k\leq n$ we sum up terms
\begin{equation}
\begin{array}{ll}
\sum_{i=1}^n\frac{\partial}{\partial x_i}a_{jk}v_{,j,k}=\sum_{i=1}^na_{jk,i}v_{,j,k}\\
\\
+a_{jk}\frac{\partial^2}{\partial x_j\partial x_k}\left( \sum_{i=1}^nv_{,i}\right)=\sum_{i=1}^na_{jk,i}v_{,j,k}
\end{array}
\end{equation}
For this reason more flexibility with respect to the second order terms as indicated in (\ref{sfu}) means that we can apply the scheme to models of incompressible Navier Stokes equations on manifolds with variable viscosity.
\end{rem}

Next we have to discuss this local Poisson equation on $U\subset {\mathbb R}^n$. Actually, in terms of a given vector field $\mathbf{v}$ we have a Dirichlet problem here of the form
\begin{equation}
\begin{array}{ll}
\Delta p(t,.)=f \mbox{ on $U$},\\
\\
p(t,x)=g \mbox{ on $\partial U$},
\end{array}
\end{equation}
where $\partial U$ is the boundary of $U$. Note that we have a family of Dirichlet problems here (one for each time $t\geq 0$), where $t$ serves as an external parameter. In our case $f$ and $g$ are defined in terms of the same function $S\left( \mathbf{ v},\nabla \mathbf{v}\right)$, where for fixed $t$ and $x\in U$
\begin{equation}
f(x):=-S^{f,U}\left( \mathbf{ v},\nabla \mathbf{v}\right)(t,x),
\end{equation}
and
\begin{equation}
g(x):=-S^{f,\partial U}\left( \mathbf{ v},\nabla \mathbf{v}\right)(t,x).
\end{equation}
Here, we may assume that domain $U$ is small enough such that the definition of $S^{f,U}$ in flat coordinates can be extended to the boundary of $\partial U$ of $U$. Another possibility is 
to define $S^{f,\partial U}\left( \mathbf{ v},\nabla \mathbf{v}\right)(t,x)$ with respect to a finite set of charts with image $V_i,~i\in I$ such that  $\partial U\subset \cup_{i\in I} V_i$ such that for fixed $t$ and $x\in \partial U\cap V_i$ we have
\begin{equation}
S^{f,\partial U\cap V_i}\left( \mathbf{ v},\nabla \mathbf{v}\right)(t,x)=S^{V_i}\left( \mathbf{ v},\nabla \mathbf{v}\right)_{ V_i}(t,x),
\end{equation}
where the charts are identical on $U\cap V_i$ for all $i\in I$. Here and in general if we consider the functional $S$ in a certain unspecified chart on $U\subset {\mathbb R}^n$ then we just write $S^U$ to indicate the functional $S$ in local coordinates at $U$. 
We may assume that the atlas for the smooth compact manifold $M$ is chosen such that each $U_i$ has a $C^{\infty}$-boundary $S$. As it is well known the solution for the pressure $p$ on the domain $U$ can be obtain by the sum of two solutions $p_1$ and $p_2$ of two Dirichlet problems with zero boundary condition for $p_1$ and zero interior condition for $p_2$ such that
\begin{equation}
p=p_1+p_2
\end{equation}
where both functions are given in terms of the Green's function $G^f_U$ (here again the superscript $f$ indicates with respect to locally flat coordinates), i.e.,
\begin{equation}
p_1=\int_Uf(y)G^f_U(x,y)dy,
\end{equation}
and 
\begin{equation}
p_2=\int_{\partial U}g(y)\partial_{\nu_y}G^f_U(x,y)d\sigma (y),
\end{equation}
where $\partial U$ denotes the boundary of $U$ and $\partial_{\nu_y}$ denotes the derivative with respect to the outward normal.
Hence locally on $U$ and with respect to locally flat coordinates we have a Leray representation of a Navier-Stokes equation of the form 
\begin{equation}\label{eqiloc1}
\begin{array}{ll}
\frac{\partial v^i}{\partial t}-\sum_{j,k=1}^na^{f,U}_{jk}(x)\frac{\partial^2 v^i}{\partial x_j\partial x_k}-\sum_{k=1}^nb^{f,U}_k(x)\frac{\partial v^i}{\partial x_k}\\
\\
+\sum_{k=1}^nv^i_{,k}v^k=\\
\\
\int_{U}\frac{\partial}{\partial x_i}G^f_U(x,y)S^{f,U}\left( \mathbf{ v},\nabla \mathbf{v}\right)(t,y)dy
+\int_{\partial U}\partial_{\nu_y}G^f_U(x,y)S^{f,U}\left( \mathbf{ v},\nabla \mathbf{v}\right)(t,y)dy,
\end{array}
\end{equation}
which can be used in order to construct global solutions of the original Navier Stokes equation on Riemannian manifolds. Recall: superscript $f$ for 'flat coordinates' on $U$, i.e., we assume that $U$ is small enough to allow for flat coordinates w.l.o.g.. Note that in (\ref{eqiloc1}) we do not impose any explicit boundary conditions on the velocity functions $v^i$. The only boundary conditions which carry infomation from other local equations are via the Poisson pressure equation and its solution with the Green function above. The reason is that this is sufficient in order to define a global scheme which converges to a global solution of the incompressible Navier Stokes equation on manifolds.
In general local coordinates we get the same equation in the form
\begin{equation}\label{navlerayeqilocU}
\begin{array}{ll}
\frac{\partial v^i}{\partial t}-\sum_{j,k=1}^na^{U}_{jk}(x)\frac{\partial^2 v^i}{\partial x_j\partial x_k}-\sum_{k=1}^nb^{U}_k(x)\frac{\partial v^i}{\partial x_k}\\
\\
+\sum_{k=1}^n\left( v^i_{,k}+\sum_{m=1}^nv^m\Gamma^{i,U}_{mk}\right)v^k=\\
\\
\int_{U}\frac{\partial}{\partial x_i}G_U(x,y)S^{U}\left( \mathbf{ v},\nabla \mathbf{v}\right)(t,y)dy
+\int_{\partial U}\partial_{\nu_y}G_U(x,y)S^{U}\left( \mathbf{ v},\nabla \mathbf{v}\right)(t,y)dy,
\end{array}
\end{equation}
where the superscript $U$ to the coefficients $a_{ij}$, $b_i$, and $\Gamma^i_{jk}$ indicate that this is the local representation in a unspecified chart $\psi: U_M\subset M\rightarrow U$. Note that in (\ref{navlerayeqilocU}) we dropped the superscript $f$ (indicating locally flat coordinates). All functions without the superscript $f$ (such as the Green function $G_U$) are defined via the coordinate transformation and the corresponding function in flat coordinates (such as $G^f_U$), i.e., all the related terms are considered to be defined in terms of the terms with superscript $f$ via coordinate transformation (this defines $a_{ij}^{U}$ $S^U$ etc.). Now if we take a finite atlas of the smooth compact manifold $M$ of the form $A:=\left\lbrace \psi_m:U_{M_m}\rightarrow U_m,~m\in I \right\rbrace $ (where we assume that each $U_m$ has a $C^{\infty}$-boundary such that the Green's function $G_{U_m}$ exists, then we obtain a 'representation' of the Navier-Stokes equation in Leray projection form on $M$ by a family of local equations each of which is of  the form (\ref{navlerayeqilocU}). Speaking more strictly. These local equations 'represent' the Navier Stokes equation on a manifold, if they all fit together not only the pressure). However, we shall see that we can define a scheme based on local equations of the form (\ref{navlerayeqilocU}) which lead to a global solution of the Navier Stokes equation on manifolds, i.e., the local equations and the pressure related source terms carray enough information in order to determine global solutions of the Navier-Stokes equation on manifold - and in this sense they 'represent' the Navier Stokes equation on Riemannian manifolds.  Note that except for the term
\begin{equation}\label{comm}
\int_{\partial U}\partial_{\nu_y}G_U(x,y)S\left( \mathbf{ v},\nabla \mathbf{v}\right)(t,y)dy
\end{equation}
all terms in (\ref{navlerayeqilocU}) are local in the sense that they are treated independently of related local terms of other equations. Indeed the term (\ref{comm}) describes the communication for a system of local equations that we are going to define.

Next we derive this global system of local equations of the form (\ref{navlerayeqilocU}) using a partition of unity. Using this partition of unity we can derive localized equations where the flow of information over the boundary (the coupling of the local equations) is based on the terms (\ref{comm}). First, let us recall some facts concerning partitions of unity on smooth manifolds. We have
\begin{prop}
Let $M$ be a compact manifold  and let $(W_i)_{i\in I}$ be an open cover of $M$. Then there exists an open cover $(V_j)_{j\in J}$ subordinate to $(W_i)_{i\in I}$ and a family of real smooth functions $(\phi_j)_{j\in J}$, such that
\begin{itemize}
 \item[i)]$\phi_j:M\rightarrow {\mathbb R}$ , where $\mbox{supp}(\phi_j)\subset V_j$, and
 \item[ii)] $\phi_j(x)\geq 0$ for all $x\in V_j$, and for all $x\in M$
 \begin{equation}
 \sum_{j\in J}\phi_j(x)=1
 \end{equation}
\end{itemize}
\end{prop}
\begin{rem}
 Since $M$ is compact we may assume that the cover $(V_j)_{j\in J}$ is finite, i.e., we may assume that the index set $J$ is of finite cardinality.
\end{rem}
\begin{rem}
The above result may be generalized to manifolds which are countably unions of compact manifolds. In this paper we restrict our considerations to compact manifolds. This simplifies some steps of the proof of global existence which are more complicated in the context of non-compact manifolds.
\end{rem}
We apply the latter proposition and consider a partition of unity $(\phi_j)_{j\in J}$ subordinate to $(V_j)_{j\in J}$ where we may assume that the latter cover is such that we have a description of the Navier-Stokes equation in locally flat coordinates on each $V_j$, i.e., for each $V_j$ we have a chart $\psi_j:V_j\rightarrow \psi_j(V_j)=:U_j\subset {\mathbb R}^n$, and such that the equation takes the form (\ref{eqiloc}), or, equivalently, the form (\ref{navlerayeqilocU}) in the chosen coordinates. We shall define a system for the family $\left( v^{ij}\right)_{1\leq i\leq n,~j\in J}$, where $v^{ij}$ denotes the restriction of $v^{i}$ to the domain $U_j$ in a chart $\psi_j$. Note that for all $j\in J$ we have in local coordinates (on $U_j$)
\begin{equation}
v^{ij}(t,x)=\sum_{k\in J_j}\phi^U_{jk}(x)v^i(t,x)
\end{equation}
along with
\begin{equation}\label{Jj}
J_{j}:=\left\lbrace  k\in J|U_j\cap U_k\neq \oslash\right\rbrace ,
\end{equation}
and where for all $j,k\in J$ we define
\begin{equation}
\begin{array}{ll}
\phi_{jk}^U:\overline{U_j}\rightarrow {\mathbb R}\\
\\
\phi_{jk}^U(z)=\left\lbrace \begin{array}{ll}
\phi_k\left(\psi^{-1}_j(z)\right)~~\mbox{if}~~\psi^{-1}_j(z)\in V_k\\
\\
0~~\mbox{else}.
\end{array}\right.
\end{array}
\end{equation}
The solution of the system for the family of functions  $\left(v^{ij}\right)_{1\leq i\leq n,j\in J}$ provides enough information for reconstruction of the solution $\mathbf{v}$ of the Navier-Stokes equation on a manifold. We have a little redundancy in that family of function, but we may define the family $\left( v^{ijj}\right)_{1\leq i\leq n,j\in J}$ along with
\begin{equation}
\begin{array}{ll}
v^{ijj}:[0,\infty)\times U_j\rightarrow {\mathbb R}\\
\\
v^{ijj}(t,x)=\phi^{U}_{jj}v^{ij}(t,x)
\end{array}
\end{equation}
in order to define a family of functions where local sums represent a solution of the Navier Stokes equation in a local chart.
The latter family, as derived from the family $v^{ij}_{1\leq i\leq n,j\in J}$,  may be considered as a representation in local coordinates of the solution of the incompressible Navier-Stokes equation on a compact Riemannian manifold subordinated to a finite atlas which allows for locally flat representations.
In locally flat coordinates we have the equation
\begin{equation}\label{eqiloc1}
\begin{array}{ll}
\frac{\partial v^{ij}}{\partial t}-\sum_{j,k=1}^na^{f,U}_{jk}(x)\frac{\partial^2 v^{ij}}{\partial x_j\partial x_k}-\sum_{k=1}^nb^{f,U}_k(x)\frac{\partial v^{ij}}{\partial x_k}\\
\\
+\sum_{k=1}^nv^{ij}_{,k}v^{kj}=S^{j}_{\mbox{int},i}\left( \mathbf{ v},\nabla \mathbf{v}\right)+S^{J_j}_{\mbox{coup},i}\left( \mathbf{ v},\nabla \mathbf{v}\right),
\end{array}
\end{equation}
where
\begin{equation}
S^{j}_{\mbox{int},i}\left( \mathbf{ v},\nabla \mathbf{v}\right)=\int_{U_j}\frac{\partial}{\partial x_i}G^f_{U_j}(x,y)S^{j,U}_i\left( \mathbf{ v},\nabla \mathbf{v}\right)(t,y)dy,
\end{equation}
along with
\begin{equation}\label{sfujj}
\begin{array}{ll}
S^{j,U}_i\left( \mathbf{ v},\nabla \mathbf{v}\right)(t,x):=-\sum_{i=1}^n\sum_{m,k=1}^na^{f,U}_{mk,i}(x)\frac{\partial^2 v^{ij}}{\partial x_m\partial x_k}\\
\\
-\sum_{i=1}^n\sum_{k=1}^nb^{f,U}_{k,i}(x)\frac{\partial v^{ij}}{\partial x_k}
+\sum_{i,k=1}^nv^{ij}_{,k}v^{kj}_{,i},
\end{array}
\end{equation}
and the coupling term $S^{J_j}_{\mbox{coup},i}\left( \mathbf{ v},\nabla \mathbf{v}\right)$ is defined below. Recall that in a usual interpretation of the Laplacian on manifolds in locally flat coordinates the spatial derivative of the second order and first order coefficients $a^{f,U}_{jk,i}$ and $b^{f,U}_{k,i}$ are zero and the definition in (\ref{sfujj}) simplifies accordingly. However, tracking them has the convenience that we can easily retrieve the information how the solution may look like in general coordinates. Moreover, we cannot stick to the same locally flat coordinates for $v^{ij}$ if we consider the functions $v^{ip}$ on the adjacent domains $U_p$ with $p\in J_j\setminus \{j\}$ in general. It is natural to consider them in natural local coordinates, i.e. such that the line elements looks locally like $ds^2=\sum_{mp}g_{mp}dx_mdx_p$. Then we have to substitute usual derivatives by affine connections. However, on the boundary of $\partial U_j$ we may consider locally flat coordinates (w.l.o.g. we may assume that the partition of unity is fine enough to have flat coordinates in a small neighborhood of each coordinate patch $U_j$).
Next we observe that 
 \begin{equation}
 v^{ijj}=\phi^{U}_{jj}v^{ij}
 \end{equation}
has zero values on the boundary $\partial U_j$. Accordingly, in order to define $S^{J_j}_{\mbox{coup},i}$ (which denote coupling terms) we consider for $p\in J_j\setminus \left\lbrace  j\right\rbrace$ and $x\in U_p\cap \partial U_j$
\begin{equation}\label{sfujp}
\begin{array}{ll}
S^{f,jp,U}_i\left( \mathbf{ v},\nabla \mathbf{v}\right)(t,x):=-\sum_{i=1}^n\sum_{m,k=1}^na^{f,U_p\cap \partial U_j}_{mk,i}(x)\frac{\partial^2 v^{ijp}}{\partial x_m\partial x_k}\\
\\
-\sum_{i=1}^n\sum_{k=1}^nb^{f,U_p\cap \partial U_j}_{k,i}(x)\frac{\partial v^{ijp}}{\partial x_k}
+\sum_{i,k=1}^nv^{ijp}_{,k}v^k_{,i},
\end{array}
\end{equation}
where the notation $a^{f,U_p\cap \partial U_j}_{mk}$ and $b^{f,U_p\cap \partial U_j}_{k}$ indicates that these are coefficients which are flat in $U_j$-coordinates (but not necessarily flat in coordinates which are flat in $U_p$). Hence, in a usual interpretation of the Laplacian on a manifold in locally flat coordinates they are locally constant and their partial derivatives are zero. However, denoting these partial derivatives in (\ref{sfujp}) accounts for a more general situation (possible generalisation) and reminds us how the term may be rewritten in general coordinates. 
Next we have
\begin{equation}
\begin{array}{ll}
S^{J_j}_{\mbox{coup},i}\left( \mathbf{ v},\nabla \mathbf{v}\right)=\\
\\
\sum_{p\in J_j\setminus \left\lbrace j\right\rbrace }\int_{\partial U}\partial_{\nu_y}G_U(x,y)S^{f,jp,U}\left( \mathbf{ v},\nabla \mathbf{v}\right)(t,y)dy.
\end{array}
\end{equation} 
Again, we emphasize that the boundary terms for $j=p$ can be eliminated
  since $\phi^U_{jj}\mathbf{v}^j$ becomes zero on the boundary of $U_j$.
Now we have a system of $\mbox{card}(J)$ equations of form (\ref{eqiloc1}) where  equation number $j$ is coupled with equation number $k$ if $k\in J_j$. 
\begin{rem}
The system in (\ref{eqiloc1}) is not intended as a definition of the Navier Stokes equation on manifolds. It is a tautological description of a bunch of local equations which is satisfied by the solution of the Navier Stokes equation on manifolds. From this description we shall derive a local scheme and a global scheme in order to obtain a local and a global solution of the Navier Stokes equation. Furthermore, from  (\ref{eqiloc1}) we can observe that there is some freedom for modelling the Navier Stokes equation on manifolds, as no physical law is known which describes how the local equation have to communicate. Symmetry assumptions of space are needed to determine this.
\end{rem}

In general coordinates we write $S^{jp,U}_i=S^{f,jp,U}$, i.e., we drop the superscript $f$.
In this representation the derivatives if the coefficient functions across the boundary $\partial U_j$ are no longer zero in general. Indeed, they describe some part of the flow of informations or the coupling between the local Navier Stokes equations for $v^{ij}$.

This leads to the idea that we may define a time-local scheme by solving equations for functions $v^{ij,m}$ which approximate $v^{ij}$ for $1\leq i\leq n$ and $j\in J$,  and where we take the coupling information from the previous iteration step $m-1$. More precisely, for $j\in J$, $1\leq i\leq n$, and $m\geq 1$  define recursively  
\begin{equation}\label{eqiloc1m}
\begin{array}{ll}
\frac{\partial v^{ij,m}}{\partial t}-\sum_{j,k=1}^na^{f,U}_{jk}(x)\frac{\partial^2 v^{ij,m}}{\partial x_j\partial x_k}-\sum_{k=1}^nb^{f,U}_k(x)\frac{\partial v^{ij,m}}{\partial x_k}\\
\\
+\sum_{k=1}^nv^{ij,m}_{,k}v^{kj,m}\\
\\
=S^{j,m}_{\mbox{int},i}\left( \mathbf{ v},\nabla \mathbf{v}\right)+S^{J_j,m-1}_{\mbox{coup}}\left( \mathbf{ v},\nabla \mathbf{v}\right),
\end{array}
\end{equation}
where
\begin{equation}
S^{j,m}_{\mbox{int},i}\left( \mathbf{ v},\nabla \mathbf{v}\right)=\int_{U_j}\frac{\partial}{\partial x_i}G^f_{U_j}(x,y)S^{j,m,U}_i\left( \mathbf{ v},\nabla \mathbf{v}\right)(t,y)dy,
\end{equation}
along with
\begin{equation}\label{sfujjm}
\begin{array}{ll}
S^{j,m,U}_i\left( \mathbf{ v},\nabla \mathbf{v}\right)(t,x):=-\sum_{i=1}^n\sum_{j,k=1}^na^{f,U}_{mk,i}(x)\frac{\partial^2 v^{ij,m}}{\partial x_j\partial x_k}\\
\\
-\sum_{i=1}^n\sum_{k=1}^nb^{f,U}_{k,i}(x)\frac{\partial v^{ij,m}}{\partial x_k}
+\sum_{i,k=1}^nv^{ij,m}_{,k}v^{kj,m}_{,i},
\end{array}
\end{equation}
and
\begin{equation}
\begin{array}{ll}
S^{J_j,m-1}_{\mbox{coup},i}\left( \mathbf{ v},\nabla \mathbf{v}\right)=\\
\\
\sum_{p\in J_j\setminus \left\lbrace j\right\rbrace }\int_{\partial U}\partial_{\nu_y}G_U(x,y)S^{f,jp,m-1,U}\left( \mathbf{ v},\nabla \mathbf{v}\right)(t,y)dy,
\end{array}
\end{equation}
along with
\begin{equation}\label{sfujpm}
\begin{array}{ll}
S^{f,jp,m-1,U}_i\left( \mathbf{ v},\nabla \mathbf{v}\right)(t,x):=-\sum_{i=1}^n\sum_{q,k=1}^na^{f,U_p}_{qk,i}(x)\frac{\partial^2 v^{ij,m-1}}{\partial x_q\partial x_k}\\
\\
-\sum_{i=1}^n\sum_{k=1}^nb^{f,U_p}_{k,i}(x)\frac{\partial v^{ij,m-1}}{\partial x_k}
+\sum_{i,k=1}^nv^{ij,m-1}_{,k}v^{kj,m-1}_{,i}.
\end{array}
\end{equation}

We may start this local scheme at some time point $t=t_0$ with $\mathbf{v}(t_0,.)=\left(v^{1}(t_0,.),\cdots,v^n(t_0,.) \right) $, if this function is known. Furthermore, we may (but latter we shall see that we do not need to) introduce boundary conditions in order to make this time-local scheme spatially global, i.e. for some $t_1>t_0$, and for all $j\in J$ we add the boundary condition
\begin{equation}\label{boundt}
v^{ij,m}(t,x)=\sum_{k\in J_j}v^{ijk,m-1}(t,x)\mbox{ for all }(t,x)\in [t_0,t_1]\times \partial U_j.
\end{equation}
For $m=1$ we may set $v^{ij,0}(t,x):=v^{i}(t_0,x)\mbox{ for all }(t,x)\in [t_0,t_1]\times \partial U_j$.  If $t=0$ and $m=1$ , then we define for $1\leq i\leq n$ and $j\in J$
\begin{equation}\label{eqiloc2m}
v^{ij,m-1}|_{m=1}=v^{ij,0}:=h^{ij}=h^i|_{U_j}.
\end{equation} 
For a horizon $t_1-t_0$ small enough we shall see that this scheme leads to a time-local fixed point iteration in classical space which defines in the limit $m\uparrow \infty$ a spatially global and time-local solution
\begin{equation}
\mathbf{v}:[t_0,t_1]\times M\rightarrow TM
\end{equation}
of the incompressible Navier-Stokes equation on the manifold $M$. The tangential bundle is local isomorphic to $U_i\times {\mathbb R}^n$, and we shall understand that we use at each $x\in {\mathbb R}^n$ the isomorphism
\begin{equation}\label{isom}
T_xM\cong{\mathbb R}^n,
\end{equation}
if we refer to the $n$ components of the velocity functions. However, in general we shall consider the velocity functions in charts, and then components of velocity functions are natural.
In a first step of our proof of global regular solutions we shall show that a reformulation of the scheme defined by the equations (\ref{eqiloc1m}), and (\ref{eqiloc2m}) and with boundary conditions (\ref{boundt}) converges to the local solution of the Navier tokes equation which is valid for small time $t>0$. Note that the local right side of the Poisson equation which eliminates the pressure involves terms  of the from $\sum_{i,k=1}^nv^i_{,k}v^k_{,i}$ (in locally flat coordinates), where the difference of two consecutive iteration steps (indexed by $m$) involves localisations of
\begin{equation}
\begin{array}{ll}
\sum_{i,k=1}^nv^{ij,m-1}_{,k}v^{kj,m-1}_{,i}-\sum_{i,k=1}^nv^{ij,m-2}_{,k}v^{kj,m-2}_{,i}\\
\\
=\sum_{i,k=1}^n\left( v^{ij,m-1}_{,k}+v^{kj,m-2}_{,k}\right)\delta v^{ij,m-1}_{,k}, 
\end{array}
\end{equation}
along with $\delta v^{ij,m-1}_{,k}:=v^{kj,m-1}_{,i}-v^{kj,m-2}_{,i}$. We shall use this in order to prove local convergence.

The incompressible Navier-Stokes equation cannot be solved by a simple global fixed point iteration. Similar as in the case of the multidimensional Burgers equation  and in the case of our global scheme for the incompressible Navier-Stokes equation on the whole domain of ${\mathbb R}^n$ we choose a time-discretized scheme and construct fixed points which are local in time (cf. \cite{KHNS} and references therein). First note that it is convenient to  choose a time step size $0<\rho< 1$ and solve at each time step $l\geq 1$ for
\begin{equation}
\begin{array}{ll}
\mathbf{v}^{\rho,l}=\left(v^{\rho,l,1},\cdots, v^{\rho,l,n}\right)^T:[l-1,l]\times M\rightarrow TM
\end{array}
\end{equation}
the incompressible Navier-tokes equation on the domain $[l-1,l]\times M$ in transformed time coordinates $\tau$ with
\begin{equation}\label{timetrans}
t=\rho\tau
\end{equation}
and with initial data $\mathbf{v}^{\rho,l}(l-1,.)=\mathbf{v}^{\rho,l-1}(l-1,.)$ being the final data of the previous time step number $l-1$, where $\mathbf{v}^{\rho,1}(0,.)=\mathbf{h}(.)$. Here, as usual, the symbol $TM$ denotes the tangential bundle of the manifold $M$.
\begin{rem}
Note that in (\ref{timetrans}) we choose a time step size $\rho >0$ which is independent of the time step number $l$. In \cite{KHNS} we considered schemes with time step size of order 
$\rho_l\sim \frac{1}{l}$.
\end{rem}
However, a more interesting fact about our scheme of a Navier-Stokes equation on compact manifolds is that we can operate on classical function spaces. We shall show that we have classical solutions for the family
\begin{equation}\label{restr}
v^{ij}=v^i|_{U_j},~1\leq i\leq n,~j\in J,
\end{equation}
where the related family
\begin{equation}\label{part}
v^{ijj}=\phi_{jj}v^{ij},~1\leq i\leq n,~j\in J
\end{equation}
may be considered as a representation of a solution of the Navier Stokes Cauchy problem on compact manifolds. Note that (\ref{restr}) can be reconstructed from (\ref{part}) and vice versa.
Next we introduce a classical function space
\begin{prop}
For open and bounded $\Omega\subset {\mathbb R}^n$ and consider the function space
\begin{equation}
\begin{array}{ll}
C^m\left(\Omega\right):={\Big \{} f:\Omega \rightarrow {\mathbb R}|~\partial^{\alpha}f \mbox{ exists~for~}~|\alpha|\leq m\\
\\
\mbox{ and }\partial^{\alpha}f \mbox{ has an continuous extension to } \overline{\Omega}{\Big \}}
\end{array}
\end{equation}
where $\alpha=(\alpha_1,\cdots ,\alpha_n)$ denotes a multiindex and $\partial^{\alpha}$ denote partial derivatives with respect to this multiindex. Then the function space $C^m\left(\overline{\Omega}\right)$ with the norm
\begin{equation}
|f|_m:=|f|_{C^m\left(\overline{\Omega}\right) }:=\sum_{|\alpha|\leq m}{\big |}\partial^{\alpha}f{\big |}
\end{equation}
is a Banach space. Here,
\begin{equation}
{\big |}f{\big |}:=\sup_{x\in \Omega}|f(x)|.
\end{equation}

\end{prop}
Local convergence of the scheme for $v^{ij,p},~p\geq 0$ above implies that for given $t>0$ and for each $j\in J$ a sequence
\begin{equation}
v^{ijj,p}(t,.)=\phi^U_jv^{ij,p}(t,.)\in C^{2}(U_j)
\end{equation}
converges in $C^{2}(U_j)$. This leads to classical solutions $v^{ijj}$ in local spaces $C^{1,2}\left((t_0,t_1)\times U_j \right)$.
We shall construct a bounded classical solution $$\mathbf{v}:\left[0,\infty\right)\times M\rightarrow TM,$$ where all components $v_i$ are globally bounded with respect to the $|.|_{1,2}$ norm, i.e., the suprema up to first order time derivatives and up to second order spatial derivatives are bounded in all local charts of a smooth atlas.  
 
In order to make the scheme global we introduce another idea, a control function $\mathbf{r}$.
First we introduce a new time variable 
\begin{equation}
\tau \rightarrow \rho t,
\end{equation}
where $\rho>0$ is a small constant  to be determined, and such that each time step with respect to the new time-variable $\tau$ is of unit length. Accordingly, for each time step number $l\geq 1$ the original velocity function restricted to the time interval $[l-1,l)$ and denoted by $\mathbf{v}^l=(v^{l,1},\cdots ,v^{l,n})^T$ (recall the implicit use of identifications $T_xM\cong{\mathbb R}^n$ at each $x\in {\mathbb R}^n$) along with
\begin{equation}
\mathbf{v}^{l}:[\rho(l-1),\rho l)\times M\rightarrow TM,
\end{equation}
there exists a time-transformed velocity
\begin{equation}
\mathbf{v}^{\rho ,l}:[l-1,l)\times M\rightarrow TM
\end{equation}
with components $v^{\rho,l,i},~1\leq i\leq n$.
Next, for integers $l\geq 1$ (time-steps) we shall construct  a family of recursively defined functions $\mathbf{r}^l,~l\in {\mathbb N},~1\leq i\leq n$, where
\begin{equation}
\begin{array}{ll}
\mathbf{r}^l:[l-1,l]\times M\rightarrow TM,
\end{array}
\end{equation}
and consider for $1\leq i\leq n$ and $j\in J$ the equation system for the functions
\begin{equation}
v^{r,\rho,l,ij}=v^{\rho ,l,ij}+r^{l,ij}.
\end{equation}
The series $\left( \mathbf{r}^l\right)_l$ defines a global control function $\mathbf{r}$ which looks locally on $U_j$ as a $n$-tuple $\left( r^{l,1j},\cdots, r^{l,nj}\right)^T$ on each subdomain $[l-1,l]\times U_j$ (in a chart). These local representations determine the global control function 
\begin{equation}
\begin{array}{ll}
\mathbf{r}:\left[ 0,\infty\right) \times M\rightarrow TM,
\end{array}
\end{equation}
which is designed in order to control the absolute value and the first derivatives of the functions $\mathbf{v}^{r,\rho,l}$ and of the functions $\mathbf{r}^l$ themselves.  We use the freedom we have in order to define the control function $r^{l,ij}$ in order to determine source terms of the equations for $\mathbf{v}^{r,\rho,l}$, where we define $\mathbf{r}^l$ by dynamic recursion such that $\mathbf{r}$ is a bounded function.
If we know that $\mathbf{r}$ is bounded and globally H\"{o}lder-continuous and has bounded continuous spatial derivatives of first order, and if we know that $\mathbf{v}^r=\mathbf{v}+\mathbf{r}$ is globally H\"{o}lder continuous with bounded continuous derivatives of first order, then $\mathbf{v}=\mathbf{v}^r-\mathbf{r}$ is globally H\"{o}lder-continuous with bounded derivatives of first order. Then we may look at the original 
Navier-Stokes equation and consider the first order coefficients  involving the velocity  components $v^i$ of $n$ equations to be known coefficients of $n$ linear scalar parabolic equation with a source term involving first spatial derivatives of $v^i$ which we can consider to be known, too. Classical representations of the velocity components in terms of a fundamental solution $\Gamma_v$ are then available, and this leads to classical regularity of the velocity components $v^i$. The idea for the construction of the control function $\mathbf{r}$ is as follows. We define the function $\mathbf{r}^l$ for $l\geq 1$ time step by time step where at each time step certain source terms (consumption terms) are defined in terms of the data $\mathbf{v}^{r,\rho,l-1}(l-1,.)$ and $\mathbf{r}^{l-1}(l-1,.)$ obtained at the previous time step. We have some freedom in order to define the control function ${\mathbf r}^l$ at each time step. The source terms are chosen close to $-\mathbf{v}^{r,\rho,l-1}(l-1,.)$ have no step size factor $\rho$ such that the integral over one time step ensures that the source terms control the growth and dominate the time-local growth of the controlled velocity function $\mathbf{v}^{r,l}$. Now if we write down the equation for the controlled velocity function $\mathbf{v}^r$ we get bundle of Navier-Stokes-type terms for $\mathbf{v}^r$ plus a bundle of Navier-Stokes type terms for $\mathbf{r}$ plus mixed terms which are bilinear in $\mathbf{v}^r$ and $\mathbf{r}$. Given $\mathbf{v}^{r,l-1}(l-1,.)$ and $\mathbf{r}^{l-1}(l-1,.)$ at time $\tau=l-1$ we may determine $r^{l,ij}$ via linearized equations with first order coefficients $r^{l-1,ij}$, and with some consumption terms or source terms. Another simpler possibility is to define $r^{l,ij}$ at the beginning of each time step via the indicated source terms. Next we consider the ideas of controlled schemes in more detail. 

\section{Definition of the controlled global scheme}As indicated in the introduction, we shall use the time coordinates $\tau=\rho t$ for some small $\rho$ to be determined. This leads to a factor $\rho$ for all terms except the time derivative if we replace $t$ by $\tau$. At each time step $l\geq 1$ we have to solve for $n\cdot\mbox{card}(J)$ equations for
\begin{equation}
v^{\rho,l,ij}:[l-1,l]\times U_j\rightarrow {\mathbb R},
\end{equation}
for $1\leq i\leq n$ and $j\in J$, and where 
\begin{equation}
 v^{\rho,l,ij}(\tau,.)=v^{l,ij}(t,.).
\end{equation}
This family of local functions determines a spatially global function
\begin{equation}
\mathbf{v}^{\rho,l}:[l-1,l]\times M\rightarrow TM
\end{equation}
at each time step $l\geq 1$. The superscript $\rho$ indicates that we are considering time coordinates $\tau$ related to a time-step size $\rho$, and the number $l\geq 1$ indicates the time step number.
The equation for $v^{\rho,l,ij}$ is based on (\ref{eqiloc1}) and is of the form
\begin{equation}\label{eqiloc1lrho}
\begin{array}{ll}
\frac{\partial v^{\rho,l,ij}}{\partial \tau}-\sum_{j,k=1}^na^{f,U}_{jk}(x)\frac{\partial^2 v^{\rho,l,ij}}{\partial x_j\partial x_k}-\sum_{k=1}^nb^{f,U}_k(x)\frac{\partial v^{\rho,l,ij}}{\partial x_k}\\
\\
+\sum_{k=1}^nv^{\rho,l,ij}_{,k}v^{\rho,l,kj}=\rho S^{j}_{\mbox{int},i}\left( \mathbf{ v}^{\rho,l},\nabla \mathbf{v}^{\rho,l}\right)+\rho S^{J_j}_{\mbox{coup},i}\left( \mathbf{ v}^{\rho,l},\nabla \mathbf{v}^{\rho,l}\right).
\end{array}
\end{equation}
 At each time step $l\geq 1$ the initial values of the functions $v^{l,\rho,ijj}(l-1,.)$ are the final values of the previous time step, i.e.,
\begin{equation}
v^{\rho,l,ij}(l-1,.)=v^{\rho,l-1,ij}(l-1,.).
\end{equation}
Note that the local scheme in (\ref{eqiloc1m}) represents a family of local Navier-Stokes equations in Leray projection form with an additional coupling term. In order to solve this system for each $m\geq 1$ we need an additional iteration. At each iteration step $m$ we first use the information $v^{ij,m-1},~1\leq i\leq n,~j\in J$ from the preceding iteration step $m-1$ in order to determine the coupling term and then we solve iteratively 
linear equations at each substep $p\geq 1$ for functions $v^{ij,m,p}$ approximating $v^{ij,m}$. At approximation substep $p$ of stage $m$ of our construction the functions $v^{ij,m,p}$ functions  solve linear equations of the form
\begin{equation}\label{eqiloc1mk}
\begin{array}{ll}
\frac{\partial v^{\rho,l,ij,m,p}}{\partial \tau}-\rho\sum_{q,k=1}^na^{f,U}_{qk}(x)\frac{\partial^2 v^{\rho,l,ij,m,p}}{\partial x_q\partial x_k}-\rho\sum_{k=1}^nb^{f,U}_k(x)\frac{\partial v^{\rho,l,ij,m,p}}{\partial x_k}\\
\\
+\rho\sum_{k=1}^nv^{\rho,l,ij,m,p}_{,k}v^{\rho,l,k,j,m,p-1}
=\rho S^{j}_{\mbox{int},i}
\left( \mathbf{ v}^{\rho,l,j,m,p-1},\nabla \mathbf{v}^{\rho,l,j,m,p-1}\right)\\
\\ +
\rho S^{J_j}_{\mbox{coup},i}\left( \mathbf{ v}^{\rho,l,ij,m-1},\nabla \mathbf{v}
^{\rho,l,ij,m-1}\right),
\end{array}
\end{equation}
where for $p=1$ we have $v^{\rho,l,kj,m,p-1}=v^{\rho,l,kj,m,0}:=v^{\rho,l,kj,m-1}$. In the following we also write
\begin{equation}
S^{l,j,m,p-1}_{\mbox{int},i}
\left( \mathbf{ v},\nabla \mathbf{v}\right):=
S^{j}_{\mbox{int},i}
\left( \mathbf{ v}^{\rho,l,j,m,p-1},\nabla \mathbf{v}^{\rho,l,j,m,p-1}\right)
\end{equation}
and
\begin{equation}
S^{l,J_j,m-1}_{\mbox{coup},i}\left( \mathbf{ v},\nabla \mathbf{v}\right)=S^{J_j}_{\mbox{coup},i}\left( \mathbf{ v}^{\rho,l,ij,m-1},\nabla \mathbf{v}
^{\rho,l,ij,m-1}\right),
\end{equation}
if this is convenient.
Note that we have linearized the convection term and 'trivialized' the Leray projection term
$S^{j,m,p-1}_{\mbox{int},i}\left( \mathbf{ v},\nabla \mathbf{v}\right)$ in the sense it is defined in by the previous iteration step $p-1$ and serves as a source function of a linear parabolic equation. At each (sub-)iteration step $p$ we define
\begin{equation}
v^{\rho,l,ij,m,p}(l-1,.)=v^{\rho,l-1,ij}(l-1,.),
\end{equation}
of course. Furthermore, for each $j\in J$ we may add a boundary condition for the local problem on $[l-1,l]\times U_j$ such that for all $(\tau,x)\in [l-1,l]\times \partial U_j$  the restriction $v^{\rho,l,ij,m,p}|_{[l-1,l]\times \partial U_j}$ of $v^{\rho,l,ij,m,p}$ to the boundary $[l-1,l]\times \partial U_j$ satisfies 
\begin{equation}\label{boundstand}
v^{\rho,l,ij,m,p}|_{[l-1,l]\times \partial U_j}(\tau,x)=\sum_{k\in J_j}v^{\rho,l,ikk,m-1}(\tau,x).
\end{equation}
Here recall that
\begin{equation}
v^{\rho,l,ikk,m-1}=\phi^{U}_kv^{\rho,l,ik,m-1},
\end{equation}
and that $J_j$ is defined in (\ref{Jj}).
Note that we use the partition of unity here in order to ensure that the latter prescription (\ref{boundstand}) is well-defined for each iteration step $m\geq 1$. We shall see later that iteration with respect to $p$ and then with respect to $m$ leads to the time-local condition that in the limit for all $j,k,q\in J$ and all $(\tau,x)\in U_k\cap U_q\cap \partial U_j\neq \oslash$ we have
\begin{equation}\label{welldefbound}
v^{\rho,l,ik}(\tau,x)=v^{\rho,l,iq}(\tau,x).
\end{equation}
 Furthermore, in case $m=1$ we have to supplement 
\begin{equation}\label{boundstand2}
v^{\rho,l,ij,m-1,p}|_{[l-1,l]\times \partial U_j}(\tau,x)=v^{\rho,l-1,ik}(\tau,x).
\end{equation}
for all $(\tau,x)\in [l-1,l]\times \partial U_j$.

The first step for a global existence proof of classical solutions of the incompressible Navier-Stokes equation on compact manifolds is to show the time-local convergence of this scheme.

The next step is to define a global scheme where the main idea is to add a control function  $\mathbf{r}$ which controls the growth of the velocity function and does have at most linear growth in time itself. The growth control is time step by time step where the definition of the control function increments $\delta \mathbf{r}^l=\mathbf{r}^l-\mathbf{r}^{l-1}$ depends on the data $\mathbf{v}^{r,l-1}(l-1,.)$. The freedom of choice in the control function we have allows us to define the control function increments close to source terms of the equation for the controlled velocity function which have a damping effect on the growth of the latter function. For each time step $l\geq 1$ we shall write down the system for $\mathbf{v}^{r,\rho,l}:=\mathbf{v}^{\rho,l}+\mathbf{r}^l$ on $[l-1,l]\times M$. We have some freedom to choose the control function $\mathbf{r}^l$. There are several possible strategies. Maybe the most simple one is the following: at the beginning of time step $l$ the controlled velocity function $\mathbf{v}^{r,\rho,l-1}(l-1,.)$ and the control function are given. We may construct the local solution of the uncontrolled Navier Stokes equation problem on the domain $[l-1,l]\times M$ with these data. Then we may define a control function increment
\begin{equation}\label{controlincrement}
\delta \mathbf{r}^l=\mathbf{r}^l-\mathbf{r}^{l-1}
\end{equation}
which will allow us to control the growth of the controlled velocity value function $\mathbf{v}^{r,\rho,l}$ at time step $l$. At the same time the control function increment in (\ref{controlincrement}) should be bounded - e.g. in all local chart representations the components of the control functions are bounded by $1$.
 \begin{equation}
{\big |}{\mathbf v}^r{\big |}^n_{C\left(\left(0,T\right),H^m\left(M\right)\right)  }\leq C,
\end{equation}
and 
\begin{equation}
{\big |}{\mathbf r}{\big |}^n_{C\left(\left(0,T\right),H^m\left(M\right)\right)  }\leq C+CT.
\end{equation}
This implies a linear upper bound in time for $\mathbf{v}$, of course, i.e., for generic $C>0$ we also have for $\mathbf{v}=\mathbf{v}^r-\mathbf{r}$ the upper bound
\begin{equation}
{\big |}{\mathbf v}{\big |}^n_{C\left(\left(0,T\right),H^m\left(M\right) \right) }\leq C+CT.
\end{equation}
This is sufficient in order to prove global regular existence, and our argument below is designed in order to obtain this result and a related result with respect to the norm ${\big |}.{\big |}^n_{C\left(\left(0,T\right),H^m\left(M\right)\right)}$. Here the latter norm may be defined using a finite atlas of local charts with image $U_j\subset {\mathbb R}^n$, and the upperscipt $n$ indicates that in a local chart we have a vector valued function and may take the maximum over the norms ${\big |}.{\big |}_{C\left(\left(0,T\right),H^m\left(U_j\right)\right)}$ for each component of the vector and all indices $j\in J$ of the finite atlas. 
 
You may also look at this in the following alternative way: we may define $\mathbf{r}^l$ via a linear equation with a right side source term $\mathbf{\phi}^l$ which involves the data $\mathbf{v}^{r,\rho,l-1}(l-1,.)$ and $\mathbf{r}^{l-1}(l-1,.)$ obtained at the previous step. The linearzed equation for $\mathbf{r}^l$ is close to a full nonlinear equation with source term right side $\mathbf{\phi}^l$ which we may derive (or construct) from the equation for the controlled velocity function $\mathbf{v}^{r,\rho,l}$ at time step $l$. We may plug in our equation for $\mathbf{r}^l$ into the right side of the equation for $\mathbf{v}^{r,\rho,l}$ in order to show that the growth of $\mathbf{v}^{r,\rho,l}$ is uniformly bounded independently of the time-step number $l\geq 1$. The construction of the iteration scheme describes the road on which we proceed in order to prove the global existence. Further  comment will be made later on.

The equation for the controlled velocity function, i.e., for $v^{r,\rho,l,ijmp}:[l-1,l]\times U_j\rightarrow {\mathbb R}$ follows from
(\ref{eqiloc1mk}) via the definition
\begin{equation}
v^{r,\rho,l,ij,m,p}=v^{\rho,l,ij,m,p}+r^{l,ij}.
\end{equation}
As we said the family of control functions $r^{l,ij},~1\leq i\leq n,~j\in J$ may be defined directly at the beginning of time step $l\geq 1$ or by a linear equation in a first substep of time step $l$ and is fixed then.
\begin{rem}
In the alternative view we choose linear equations for the control functions $r^{l,ij}$ because we have classical semi-explicit representations for these equations and do not need to set up an additional iteration scheme for the control function. On the other hand the linearized equation is lose to a nonlinear equation for the control function with a source term right side which may  be suggested by the equation for $\mathbf{v}^{r}$. However, the most simple point of view is to solve a local uncontrolled Navier Stokes equation at each time step $l$ with controlled data $\mathbf{v}^{r,\rho,l-1}(l-1,.)$ from the previous time step and then define the control function increment $\delta \mathbf{r}^l$ appropriately in order to control the growth of the controlled velocity function at time step $l$. 
\end{rem}

 Therefore, in this construction $r^{l,ij}$ bears no iteration index $m$ and no subiteration index $p$.  We shall consider various possible definitions of the control functions   $r^{l,ij},~1\leq i\leq n,~j\in J$ below. The different possibilities of definitions of control functions can have an effect on the the definition of scheme for the controlled velocity functions as well. We next define the main possibilities of a scheme for the controlled velocity functions and then we shall consider corresponding definitions of the control functions. 
 
\begin{itemize}   

\item[i)] We can define a global controlled scheme without solving the equation for $\mathbf{v}^{r,\rho,l}$. Instead we just solve a locally uncontrolled Navier Stokes equation with controlled velocity data of the previous time step. At the beginning of time step $l\geq 1$ we have computed $\mathbf{v}^{r,\rho,l-1}(l-1,)$ or we have the data $\mathbf{h}$. Then for all $1\leq i\leq n$ and all $j\in J$ we locally solve the equation  
  \begin{equation}\label{eqiloc1lrhor}
\begin{array}{ll}
\frac{\partial v^{r^{l-1},\rho,l,ij}}{\partial \tau}-\sum_{q,k=1}^na^{f,U}_{qk}(x)\frac{\partial^2 v^{r^{l-1},\rho,l,ij}}{\partial x_q\partial x_k}-\sum_{k=1}^nb^{f,U}_k(x)\frac{\partial v^{r^{l-1},\rho,l,ij}}{\partial x_k}\\
\\
+\sum_{k=1}^nv^{\rho,l,ij}_{,k}v^{r^{l-1},\rho,l,kj}=\rho S^{j}_{\mbox{int},i}\left( \mathbf{ v}^{r^{l-1},\rho,l},\nabla \mathbf{v}^{r^{l-1},\rho,l}\right)\\
\\
+\rho S^{J_j}_{\mbox{coup},i}\left( \mathbf{ v}^{r^{l-1},\rho,l},\nabla \mathbf{v}^{r^{l-1},\rho,l}\right),
\end{array}
\end{equation}
with data
\begin{equation}
\begin{array}{ll}
v^{r^{l-1},\rho,l,ij}(l-1,.)=v^{r^{l-1},\rho,l-1,ij}(l-1,.)\\
\\
:=v^{\rho,l-1,ij}(l-1,.)+r^{l-1,ij}(l-1,.)=v^{r,\rho,l-1,ij}(l-1,.).
\end{array}
\end{equation}
We may solve this local equation iteratively where we may add boundary conditions as in (\ref{boundt}), but we can also work without these boundary conditions, i.e. it suffices to impose boundary conditions for the Poisson equations coded in the coupling term. Without boundary conditions the solutions may be not unique but it will become unique on a global scale when all local equations fit together. This is done by an iteration scheme (cf. proof of the main theorem).  This procedure leads to a time-local and spatially global solution $\mathbf{v}^{r,\rho,l}$ on $[l-1,l]\times M$. 

\item[ii)]  We can also write down the local equations for the controlled velocity functions, and work with these equations directly. However, this approach is formally more complicated. The iteration scheme for $v^{r,\rho,l,ij,m,p}$ then becomes
\begin{equation}\label{eqiloc1mkr}
\begin{array}{ll}
\frac{\partial v^{r,\rho,l,ij,m,p}}{\partial \tau}-\rho\sum_{q,k=1}^na^{f,U}_{qk}(x)\frac{\partial^2 v^{r,\rho,l,ij,m,p}}{\partial x_q\partial x_k}-\rho\sum_{k=1}^nb^{f,U}_k(x)\frac{\partial v^{r,\rho,l,ij,m,p}}{\partial x_k}\\
\\
+\rho\sum_{k=1}^nv^{r,\rho,l,ij,m}_{,k}v^{r,\rho,l,kj,m,p-1}\\
\\
=\frac{\partial r^{l,ij}}{\partial \tau}-\rho\sum_{q,k=1}^na^{f,U}_{qk}(x)\frac{\partial^2 r^{l,ij}}{\partial x_q\partial x_k}-\rho\sum_{k=1}^nb^{f,U}_k(x)\frac{\partial r^{l,ij}}{\partial x_k}\\
\\
+\rho\sum_{k=1}^nr^{l,ij}_{,k}v^{r,\rho,l,kj,m,p-1}+\rho\sum_{k=1}^nv^{r,\rho,l,ij,m,p-1}_{,k}r^{l,kj}\\
\\
+\rho\sum_{k=1}^nr^{l,ij}_{,k}r^{l,kj}+
\rho S^{r,l,j,m,p-1}_{\mbox{int},i}\left( \mathbf{ v},\nabla \mathbf{v}\right)+\rho S^{r,l,J_j,m-1}_{\mbox{coup},i}\left( \mathbf{ v},\nabla \mathbf{v}\right),
\end{array}
\end{equation}
where
\begin{equation}
 S^{r,l,j,m,p-1}_{\mbox{int},i}\left( \mathbf{ v},\nabla \mathbf{v}\right)=S^{l,j,m,p-1}_{\mbox{int},i}\left( \mathbf{ v}+\mathbf{r},\nabla \mathbf{v}+\nabla\mathbf{r}\right),
\end{equation}
\begin{equation}
S^{r,l,J_j,m-1}_{\mbox{coup},i}\left( \mathbf{ v},\nabla \mathbf{v}\right)=S^{l,J_j,m-1}_{\mbox{coup},i}\left( \mathbf{ v}+\mathbf{r},\nabla \mathbf{v}+\nabla\mathbf{r}\right)
\end{equation}
and where
\begin{equation}
v^{r,\rho,l,ij,m,p}(l-1,.)=v^{r,\rho,l-1,ij}(l-1,.).
\end{equation}
Again, for each $j\in J$ we may impose boundary conditions, i.e.,
for all $(\tau,x)\in [l-1,l]\times \partial U_j$ the restriction $v^{r,\rho,l,ij,m,p}|_{[l-1,l]\times \partial U_j}$ of $v^{r,\rho,l,ij,m,p}$ to the boundary $[l-1,l]\times \partial U_j$ satisfies
\begin{equation}\label{boundstandr}
v^{r,\rho,l,ij,m,p}|_{[l-1,l]\times \partial U_j}(\tau,x)=\sum_{k\in J_j}v^{r,\rho,l,ikk,m-1}(\tau,x).
\end{equation}
In case $m=1$ we define 
\begin{equation}\label{boundstand2}
v^{r,\rho,l,ij,m-1,p}|_{[l-1,l]\times \partial U_j}(\tau,x)=v^{r,\rho,l-1,ik}(\tau,x)
\end{equation}
for all $(\tau,x)\in [l-1,l]\times \partial U_j$.
The global scheme is an iteration in time of a local iteration scheme of this controlled equation {\it and} a definition of the control function $\mathbf{r}^l$. At each time step $l\geq 1$ the functions  $v^{r,\rho,l-1,ij}(l-1,.)$ and $r^{l-1,ij}(l-1,.)$ are known for $1\leq i\leq n$ and $j\in J$. For $l=1$ we may set $r^{l-1,ij}\equiv 0$ and
\begin{equation}
v^{r,\rho,0,ij}(l-1,.)=v^{\rho,0,ij}(l-1,.)=h^{ij}(.).
\end{equation}
\end{itemize}

Next we define the control functions $r^{l,ij}$ for $1\leq i\leq n$ and $j\in J$ by the following short list. At the beginning of time step $l$ the functions $v^{r,\rho,l-1,ij}(l-1,.),~1\leq i\leq n,~j\in J$ and $r^{l-1,ij}(l-1,.)~1\leq i\leq n,~j\in J$ are known. Especially, for this reason it is sufficient to define the control function increments $\delta r^{l,ij}=r^{l,ij}-r^{l-1,ij}(l-1,.),~1\leq i\leq n,~j\in J$ in order to determine the control function at the next time $l$. All the following definitions make sense only if we choose the time step size small enough. We may even have a deceasing time-step size which depends on the time step number $l$ and is of order $\rho_l\sim \frac{1}{l}$. This would still be sufficient in order to render the scheme global. We shall first define various alternatives of control function increments, and then we shall discuss time step sizes below more explicitly in the statement of the main theorem and its proof. 

\begin{itemize}

\item[i)] Our simplest definition of the control functions increments $\delta r^{l,ij},~1\leq i\leq n,~j\in J$ is
\begin{equation}
\delta r^{l,ij}(\tau,x):=\int_{l-1}^{\tau}\left( -\frac{v^{r,\rho,l-1,ij}(l-1,y)}{C}\right) C_{U_j}p^{l,j}(s-(l-1),x-y)dyds,
\end{equation}
where $p^{l,j}$ is the fundamental solution of the local diffusion equation of the Navier Stokes equation in local coordinates on $[l-1,l]\times U_j$ (which is essentially a heat equation in locally flat coordinates). The constant $C_{U_j}$ is a normalisation constant which ensures that the local integral over t$U_j$ integrates to $1$. The use of such a density is optional in the end.  We shall observe that this definition leads to  global bound of $\mathbf{v}^{r,\rho,l}$ for all $l$, i.e., there exists a constant $C>0$ depending only on dimension $n$, data $\mathbf{h}$ and the order of multivariate derivatives $|\alpha|\leq m$, and which is independent of the time step number $l\geq 1$ such that for all $1\leq i\leq n,~j\in J$ and all $l\geq 1$ we have
\begin{equation}
\sup_{l\geq 1}\max_{1\leq i\leq n,j\in J}\sup_{(\tau,x)\in [l-1,l]\times U_j}{\big |}D^{\alpha}_xv^{r,\rho,l,ij}(\tau,x){\big |}\leq C.
\end{equation}
Furthermore, there is a linear upper bound for the control functions $r^{l,ij},~1\leq i\leq n,~j\in J$, i.e. we have
\begin{equation}
\max_{1\leq i\leq n,j\in J}\sup_{(\tau,x)\in [l-1,l]\times U_j}{\big |}D^{\alpha}_xr^{l,ij}(\tau,x){\big |}\leq C+Cl.
\end{equation}
This result is sufficient in order to prove global regular existence of the solution $\mathbf{v}=\mathbf{v}^r-\mathbf{r}$ of the incompressible Navier Stokes equation. It leads to an global linear bound of this solution. This result can be sharpened if we look at a more involved definition of the control functions.
\item[ii)]  the control functions increments $\delta r^{l,ij},~1\leq i\leq n,~j\in J$ can be defined independently of the locally uncontrolled velocity functions $v^{r^{l-1},\rho,l,ij}$ just in terms of the data $r^{l-1,ij}(l-1,.)$ and $v^{r,\rho,l-1,ij}(l-1,.)$ of the previous time step. For all $1\leq i\leq n$ and all $j\in J$, and $\tau\in [l-1,l]$ and $x\in U_j$ we consider a short list of two possible definition of local control function values $\delta r^{l,ij}(l,x)=r^{l,ij}(l,x)-r^{l-1,ij}(l-1,x)$. Next we 
define a property $P$ which has the effect of a switch. According to the situation whether the property $P$ holds or does not hold at the end of time step $l-1$ we choose the control function increment at time step $l$. As long as we are in the situation of item i) with the additional condition that we have the upper bound $C>0$ for the modulus of the control functions (and for multivariate spatial derivatives of )  $r^{l,ij}(l,.)$ we continue to define the control function as in item i). However, if the modulus of the control function (or some multivariate derivative) exceeds $C$, then we define the control function increment in terms of the negative data of the control function itself.  Let
\begin{equation}
M^{l-1,\alpha}_r:=\max_{1\leq i\leq n,~j\in J}\sup_{x\in U_j}{\big |}D^{\alpha}_xr^{l,ij}(l-1,x){\big |},
\end{equation}
As the Riemannian manifold is compact, these maxima are obtained for some $x\in U_{j_0}$, where $U_j$ is the image of a chart with domain $V_{j_0}\subset M$. 
Let 
\begin{equation}\label{Pcond}
\begin{array}{ll}
\mbox{P}:~ M^{l-1,\alpha}_r\leq C \mbox{for all $\alpha$ with }|\alpha|\leq 2.
\end{array}
\end{equation}
Then we simply write $P$ if the condition $\mbox{P}$ in (\ref{Pcond}) holds and $\mbox{non-P}$ if the condition $\mbox{P}$ in (\ref{Pcond}) does not hold.
\begin{equation}
\begin{array}{ll}
\delta r^{l,ij}(\tau,x):=\\
\\
\left\lbrace \begin{array}{ll}
				\int_{l-1}^{\tau}\left( -\frac{v^{r,\rho,l-1,ij}(l-1,y)}{C}\right)C_{U_j}p^{l,j}(\tau-(l-1),x-y)dy~\mbox{if $\mbox{P}$}\\
				\\
				\int_{l-1}^{\tau}\left( -\frac{r^{l-1,ij}(l-1,y)}{C}\right)C_{U_j}p^{l,j}(\tau-(l-1),x-y)dy~\mbox{if $\mbox{non-P}$}.
                                \end{array}\right.
\end{array}
\end{equation}
For the source terms involved we also use the notation
\begin{equation}
\phi^{v,l,ij}=-\frac{v^{r,\rho,l-1,ij}(l-1,y)}{C},
\end{equation}
and
\begin{equation}
\phi^{r,l,ij}=-\frac{r^{l-1,ij}(l-1,y)}{C}.
\end{equation}

For $l=1$ we take as data for the control function
\begin{equation}
r^{l-1,ij}(l-1,.)=r^{0,ij}(0,.)=\frac{h^{ij}(.)}{C},
\end{equation}
for all $1\leq i\leq n$ and $j\in J$. Note that the control function data have the same sign as the controlled velocity function data at the first time step.

\item[iii)] All the preceding definitions of a control function can be realized in a related context, where we solve equations for the control functions with a source term on the right side which is then essentially defined to be the integrand of the control functions increments in item i) or alternatively in item ii). 
Note that the terms $v^{r,\rho,l-1,ij}$ and $r^{l-1,ij}$ in (\ref{eqiloc1mkr*}) below are abbreviations for the same functions evaluated at $l-1$, i.e. they equal $v^{r,\rho,l-1,ij}(l-1,.)$ and $r^{l-1,ijj}(l-1,.)$.
The defining equation for $r^{l,ij,m}$ (iteration index $m$ is in order to define a spatially global solution of this linear equation on the whole manifold $M$) is:

\begin{equation}\label{eqiloc1mkr*}
\begin{array}{ll}
\frac{\partial r^{l,ij,m}}{\partial \tau}-\rho\sum_{q,k=1}^na^{f,U}_{qk}(x)\frac{\partial^2 r^{l,ij,m}}{\partial x_q\partial x_k}-\rho\sum_{k=1}^nb^{f,U}_k(x)\frac{\partial r^{l,ij,m}}{\partial x_k}\\
\\
+\rho\sum_{k=1}^nr^{l,ij,m}_{,k}r^{l-1,kjm}\\
\\
=\phi^{l,ij}-\rho\sum_{k=1}^nr^{l,ij,m}_{,k}v^{r,\rho,l-1,kj}-\rho\sum_{k=1}^nv^{r,\rho,l-1,ij}_{,k}r^{l-1,kj,m}\\
\\
-
\rho S^{r,l-1,jm}_{\mbox{int},i}\left( \mathbf{ v},\nabla \mathbf{v}\right)-\rho S^{r,l-1,J_j,m}_{\mbox{coup},i}\left( \mathbf{ v},\nabla \mathbf{v}\right),
\end{array}
\end{equation} 
where we have to define the latter to terms and the source terms $\phi^{l,ij}$ (the idea of the definition of the latter has been indicated above). Here, the upper index $l-1$ indicates that we evaluate related functions at data obtained from the previous time step. Furthermore,
\begin{equation}
S^{r,l-1,j,m}_{\mbox{int},i}\left( \mathbf{ v},\nabla \mathbf{v}\right)=S^{j}_{\mbox{int},i}\left( \mathbf{ v}^{r,\rho,l-1,ij,m},\nabla \mathbf{v}^{r,\rho,l-1,ij,m}\right).
\end{equation}
The initial conditions for (\ref{eqiloc1mkr*}) are
\begin{equation}
r^{l,ij,m}(l-1,.)=r^{l-1,ij}(l-1,.).
\end{equation}
Since we want to construct a global solution
\begin{equation}
\mathbf{r}^l:[l-1,l]\times M\rightarrow TM
\end{equation}
to this linear parabolic equation, we have to ensure that the local solutions match on the boundaries $\partial U_j,~j\in J$. As in the case of a controlled velocity function we shall see that this will be  obtained automatically by the communication of all local equations via the local Lery projection terms which we have defined via  Green's functions. We can also ensure this by boundary conditions in the same spirit as before, i.e., we define for all $\tau,x)\in [l-1,l]\times \partial U_j$
\begin{equation}
r^{l,ij,m}(\tau ,x)=\sum_{k\in J_j}r^{l,ikk(m-1)}(\tau ,x).
\end{equation}
 Then we can define the source terms $\phi^{l,ij}$ related to the prescriptions in item i) or item ii). Following the ideas of item i) we define
 \begin{equation}
\phi^{l,ij}(\tau,x):= -\frac{v^{r,\rho,l-1,ij}(l-1,x)}{C}.
\end{equation}
Following the ideas of item ii) we would define
 \begin{equation}
\begin{array}{ll}
\delta r^{l,ij}(\tau,x):=\\
\\
\left\lbrace \begin{array}{ll}
				 -\frac{v^{r,\rho,l-1,ij}(l-1,x)}{C}~\mbox{if $\overline{P}$}\\
				\\
				 -\frac{r^{l-1,ij}(l-1,x)}{C}~\mbox{if $\mbox{non}-\overline{P}$}.
                                \end{array}\right.,
\end{array}
\end{equation}
where $\overline{P}$ is a property which is defined analogously as the property $P$.
\end{itemize}

The constant $C\gg 1$ will be chosen below as will the time step size where we may choose a step size of order
\begin{equation}
\rho\sim \frac{1}{C^3}.
\end{equation}
This means for $1\leq i\leq n$ and $j\in J$ that both functions $\phi^{v,l,ij}$ $\phi^{r,l,ij}$ can dominate all terms which have a factor $\rho$ concerning the growth behavior of the controlled velocity function and of the control function from time step $l-1$ to time step $l$ respectively.
Next to dimension the constant $C>0$ depends only on the initial data $\mathbf{h}$, constants of the manifold (determined by the Christoffel symbols or by the Riemann tensor), the first and second order coefficients of the local equations, and the order of multivariate derivatives for which we want to construct an upper bound.
Next we formulate the complete scheme in the simplified version as in $i)$ and $ii)$ above. At the first time step $l=1$ we set 
\begin{equation}
r^{l-1,ij}(l-1,.)=r^{0,ij}(0,.)\equiv \frac{h^{ij}}{C.}
\end{equation}
 for all $1\leq i\leq n$ and $j\in J$. At time step $l=1$ the data of the controlled velocity functions are then given by 
\begin{equation}\label{0data}
 v^{r,\rho,l-1,ij}=v^{r,\rho,0,ij}=\left(1+\frac{1}{C}\right) h^{ij}.
 \end{equation}
 \begin{rem}
Alternatively, we could define 
\begin{equation}
r^{l-1,ij}(l-1,.)=r^{0,ij}(0,.)\equiv 0
\end{equation}
and, accordingly,
\begin{equation}\label{0data*}
 v^{r,\rho,l-1,ij}=v^{r,\rho,0,ij}= h^{ij}.
 \end{equation}
 However, we prefer to give a description where the scheme for the first time step looks similar as the scheme for the later time steps $l\geq 2$.
\end{rem}
Next we compute a local solution $v^{r^{l-1},\rho,l,ij}=v^{r^{0},\rho,l,ij}:=v^{r,\rho,0,ij}$ of the uncontrolled Navier Stokes equation on $[0,1]\times M$ with data (\ref{0data}).
This is done by a spatially global iteration of spatially local equation with an iteration index $m$ where each local equation (fixed $j\in J$) is solved by a subiteration with another iteration index $p$. For $l=1$ and each $m\geq 1$ we solve $\mbox{card}(J)$ local Navier-Stokes equation of the form
\begin{equation}\label{eqiloc1m1stepsub}
\begin{array}{ll}
\frac{\partial v^{r^0,\rho,1,ij,m,p}}{\partial \tau}-\rho\sum_{j,k=1}^na^{f,U}_{jk}(x)\frac{\partial^2 v^{r^0,\rho,1,ij,m,p}}{\partial x_j\partial x_k}-\rho\sum_{k=1}^nb^{f,U}_k(x)\frac{\partial v^{r^0,\rho,1,ij,m,p}}{\partial x_k}\\
\\
+\rho\sum_{k=1}^nv^{r^0,\rho,1,ij,m,p}_{,k}v^{r^0,\rho,1,kj,m,p-1}\\
\\
=\rho S^{j}_{\mbox{int},i}\left( \mathbf{ v}^{r^0,\rho,1,m,p-1},\nabla \mathbf{v}^{r^{0},\rho,1,m,p-1}\right)
+\rho S^{J_j}_{\mbox{coup},i}\left( \mathbf{ v}^{r^0,\rho,1,m-1},
\nabla \mathbf{ v}^{r^0,\rho,1,m-1}\right),
\end{array}
\end{equation}
with initial data
\begin{equation}
v^{r^0,\rho,1,ij,m,p}(0,.)=\left(1+\frac{1}{C}\right) h^{ij}(.).
\end{equation}

At this first time step for the subiteration with iteration index $p$ we may impose the boundary condition
\begin{equation}\label{boundstand0p}
v^{r,\rho,1,ij,m,p}|_{[0,1]\times \partial U_j}(\tau,x)=\sum_{k\in J_j}v^{r,\rho,1,ikk,m-1}(\tau,x),
\end{equation}
where the double superscript $kk$ on the right side of (\ref{boundstand0}) indicates the involvement of the partition of unity as explained above. However, imposing these boundary conditions is not necessary as we shall observe.

For a step size $\rho>0$ which is small enough the sequences
\begin{equation}
\left( v^{r^0,\rho,1,ij,m,p}\right)_{p\in {\mathbb N}}
\end{equation}
converge classically to a limit
\begin{equation}
v^{r^0,\rho,1,ij,m}:=\lim_{p\uparrow \infty}v^{r^0,\rho,1,ij,m,p}\in C^{1,2}\left([0,1]\times U_j\right) 
\end{equation}
for all $1\leq i\leq n$ and $j\in J$.
For the functional sequence $\left( v^{r^0,\rho,1,ij,m}\right)_{m\in {\mathbb N}}$ we get a spatially global iteration scheme
\begin{equation}\label{eqiloc1m1stepm}
\begin{array}{ll}
\frac{\partial v^{r^0,\rho,1,ij,m}}{\partial \tau}-\rho\sum_{j,k=1}^na^{f,U}_{jk}(x)\frac{\partial^2 v^{r^0,\rho,1,ij,m}}{\partial x_j\partial x_k}-\rho\sum_{k=1}^nb^{f,U}_k(x)\frac{\partial v^{r^0,\rho,1,ij,m}}{\partial x_k}\\
\\
+\rho\sum_{k=1}^nv^{r^0,\rho,1,ij,m}_{,k}v^{r^0,\rho,1,kj,m}\\
\\
=\rho S^{j}_{\mbox{int},i}\left( \mathbf{ v}^{r^0,\rho,1,m},\nabla \mathbf{v}^{r^0,\rho,1,m}\right)+\rho S^{J_j}_{\mbox{coup}}\left( \mathbf{ v}^{r^0,\rho,1,ij,m-1},\nabla \mathbf{v}^{r^0,\rho,1,ij,m-1}\right),
\end{array}
\end{equation}
with initial data
\begin{equation}\label{l1init}
v^{r^0,\rho,1,ij,m}(0,.)=\left(1-\frac{1}{C} \right) h^{ij}(.).
\end{equation}
Again it is optional to impose a boundary condition of the form
\begin{equation}\label{boundstand0}
v^{r^0,\rho,1,ij,m}|_{[0,1]\times \partial U_j}(\tau,x)=\sum_{k\in J_j}v^{r^0,\rho,1,ikk,m-1}(\tau,x),
\end{equation}
and where for $m=1$ the initial data may be used to initialize the boundary conditions.

For a step size $\rho>0$ which is small enough the sequences
\begin{equation}
\left( v^{r^0,\rho,1,ij,m}\right)_{m\in {\mathbb N}}
\end{equation}
converge to a classical limit
\begin{equation}
v^{r^0,\rho,1,ij}:=\lim_{m\uparrow \infty}v^{r^{l-1},\rho,1,ij,m}\in C^{1,2}\left([0,1]\times U_j\right) 
\end{equation}
for all $1\leq i\leq n$ and $j\in J$.
We shall see that for small $\rho>0$ this is a time-local fixed point iteration which leads to a spatially global and time-local solution
\begin{equation}
\mathbf{v}^{r^0,\rho,1}:[0,1]\times M\rightarrow TM.
\end{equation}
Having computed the uncontrolled velocity function $v^{r^0,\rho,1}$ with 'controlled data'  $v^{r^0,\rho,0}$ which are given in local coordinates at the first time step by (\ref{l1init}), we define the control function increments $\delta r^{1,ij}=r^{1,ij}-r^{0,ij}=r^{1,ij}-\left(-\frac{1}{C} \right)h^{ij}$  according to item $i)$ or item $ii)$ above and define the controlled velocity functions at time step $l=1$ on $[l-1]\times U_j$ by
\begin{equation}
v^{r,\rho,1,ij}=v^{r^0,\rho,1,ij}+\delta r^{1,ij}=v^{\rho,1,ij}+r^{1,ij}.
\end{equation}
Here, $v^{\rho,1,ij},~1\leq i\leq n, j\in J$ is the solution in  of the time-local uncontrolled Navier Stokes equation with uncontrolled data $h^{ij}$ and the function $r^{1,ij}=r^{0,ij}+\delta r^{1,ij},~1\leq i\leq n, j\in J$ determine the time-local control function on $[0,1]\times M$. 
This describes the first time step.

Recursively, at the beginning of time step $l\geq 2$ the functions
\begin{equation}
r^{l-1,ij}(l-1,.),~v^{r,\rho,l-1,ij}(l-1,.),~1\leq i\leq n,~j\in J
\end{equation}
 are determined.
Then we determine a local solution $v^{r^{l-1},\rho,l,ij},~1\leq i\leq n,j\in J$ of the uncontrolled Navier Stokes equation on $[0,1]\times M$ with data
\begin{equation}
v^{r^{l-1},\rho,l,ij}(l-1,.):=v^{r,\rho,l-1,ij}(l-1,.).
\end{equation}
Again, this is done by a spatially global iteration of spatially local equation with an iteration index $m$ where each local equation (fixed $j\in J$) is solved by a subiteration with another iteration index $p$. For $l\geq 2$ and each $m\geq 1$ we solve $\mbox{card}(J)$ local Navier-Stokes equation of the form
\begin{equation}\label{eqiloc1m1stepsubl}
\begin{array}{ll}
\frac{\partial v^{r^{l-1},\rho,l,ij,m,p}}{\partial \tau}-\rho\sum_{j,k=1}^na^{f,U}_{jk}(x)\frac{\partial^2 v^{r^{l-1},\rho,l,ij,m,p}}{\partial x_j\partial x_k}-\rho\sum_{k=1}^nb^{f,U}_k(x)\frac{\partial v^{r^{l-1},\rho,l,ij,m,p}}{\partial x_k}\\
\\
+\rho\sum_{k=1}^nv^{r^{l-1},\rho,l,ij,m,p}_{,k}v^{r^{l-1},\rho,l,kj,m,p-1}\\
\\
=\rho S^{j}_{\mbox{int},i}\left( \mathbf{ v}^{r^{l-1},\rho,l-1,m,p-1},\nabla \mathbf{v}^{r^{0},\rho,l,m,p-1}\right)\\
\\
+\rho S^{J_j}_{\mbox{coup},i}\left( \mathbf{ v}^{r^{l-1},\rho,l,m-1},
\nabla \mathbf{ v}^{r^{l-1},\rho,l,m-1}\right),
\end{array}
\end{equation}
with initial data
\begin{equation}
v^{r^{l-1},\rho,l,ij,m,p}(l-1,.)=v^{r,\rho,l-1,ij}(l-1,.).
\end{equation}

Boundary conditions are optional and - if imposed- may be defined analogously as in the first time step.

For a step size $\rho>0$ which is small enough the sequences
\begin{equation}
\left( v^{r^{l-1},\rho,l,ij,m,p}\right)_{p\in {\mathbb N}}
\end{equation}
converge classically to a limit
\begin{equation}
v^{r^{l-1},\rho,l,ij,m}:=\lim_{p\uparrow \infty}v^{r^0,\rho,l,ij,m,p}\in C^{1,2}\left([0,1]\times U_j\right) 
\end{equation}
for all $1\leq i\leq n$ and $j\in J$.
For the functional sequence $\left( v^{r^{l-1},\rho,l,ij,m}\right)_{m\in {\mathbb N}}$ we get a spatially global iteration scheme
\begin{equation}\label{eqiloc1m1stepml}
\begin{array}{ll}
\frac{\partial v^{r^{l-1},\rho,l,ij,m}}{\partial \tau}-\rho\sum_{j,k=1}^na^{f,U}_{jk}(x)\frac{\partial^2 v^{r^{l-1},\rho,l,ij,m}}{\partial x_j\partial x_k}-\rho\sum_{k=1}^nb^{f,U}_k(x)\frac{\partial v^{r^{l-1},\rho,l,ij,m}}{\partial x_k}\\
\\
+\rho\sum_{k=1}^nv^{r^{l-1},\rho,l,ij,m}_{,k}v^{r^{l-1},\rho,l,kj,m}\\
\\
=\rho S^{j}_{\mbox{int},i}\left( \mathbf{ v}^{r^{l-1},\rho,l,m},\nabla \mathbf{v}^{r^{l-1},\rho,l,m}\right)+\rho S^{J_j}_{\mbox{coup}}\left( \mathbf{ v}^{r^{l-1},\rho,l,ij,m-1},\nabla \mathbf{v}^{r^{l-1},\rho,l,ij,m-1}\right),
\end{array}
\end{equation}
with initial data
\begin{equation}\label{l1init}
v^{r^{l-1},\rho,l,ij,m}(l-1,.)=v^{r^{l-1},\rho,l-1,ij}(l-1,.).
\end{equation}
Again it is optional to impose a boundary condition which are analogous as in the first time step (if imposed).

For a step size $\rho>0$ which is small enough the sequences
\begin{equation}
\left( v^{r^{l-1},\rho,1,ij,m}\right)_{m\in {\mathbb N}}
\end{equation}
converge to a classical limit
\begin{equation}
v^{r^{l-1},\rho,1,ij}:=\lim_{m\uparrow \infty}v^{r^{l-1},\rho,1,ij,m}\in C^{1,2}\left([0,1]\times U_j\right) 
\end{equation}
for all $1\leq i\leq n$ and $j\in J$.
We shall see that for small $\rho>0$ this is a time-local fixed point iteration which leads to a spatially global and time-local solution
\begin{equation}
\mathbf{v}^{r^{l-1},\rho,l}:[0,1]\times M\rightarrow TM.
\end{equation}
Having computed the uncontrolled velocity function $v^{r^{l-1},\rho,l}$ with 'controlled data'  $v^{r^{l-1},\rho,l-1}(l-1,.)$ which are given in local coordinates at the first time step by the functions $v^{r^{l-1},\rho,l,ij},~1\leq i\leq n,j\in J$, we define the control function increments $\delta r^{l,ij}=r^{l,ij}-r^{l-1,ij}$  according to item $i)$ or item $ii)$ above and define the controlled velocity functions at time step $l\geq 2$ on $[l-1]\times U_j$ by
\begin{equation}
v^{r,\rho,l,ij}=v^{r^{l-1},\rho,1,ij}+\delta r^{1,ij}=v^{\rho,l,ij}+r^{l,ij}.
\end{equation}
Here, $v^{\rho,l,ij},~1\leq i\leq n, j\in J$ is the solution in  of the time-local uncontrolled Navier Stokes equation with uncontrolled data $v^{\rho,l-1,ij}(l-1,.)$, and the function $r^{l,ij}=r^{l-1,ij}+\delta r^{l,ij},~1\leq i\leq n, j\in J$ determine the time-local control function on $[l-1,l]\times M$.

We shall see that for small $\rho>0$ this is a time-local fixed point iteration which leads to a spatially global and time-local solution
\begin{equation}
\mathbf{v}^{r,\rho,l}:[l-1,l]\times M\rightarrow TM.
\end{equation}

The following argument is essentially constructive up to the point that the choice of the time step size $\rho>0$ and the constant $C>0$ may be analysed more constructively.  This will be done in a subsequent paper. The size of $C>0$ and $\rho>0$ as an upper bound depends on the order of multivariate derivatives of the controlled velocity function and of the control function for which this upper bound is to be constructed, of course. At this analytic stage it suffices to sow that next to the order of derivatives considered, the constants $\rho$ and $C$ depend only the data $\mathbf{h}$ the viscosity $\nu$, and the coefficients $g_{ij}$ of the line element of the underlying manifold $M$.

\section{Main theorem}
We shall assume
\begin{equation}
\mathbf{h}\in C^{\infty}\left(M,TM\right), 
\end{equation}
or, in a family of local charts $\psi_j:V_j\rightarrow U_j\subset {\mathbb R}^n,~j\in J$ covering the manifold $M$ we have components $h^{ijj}\in C^ {\infty}\left(U_j\right)$ for all $1\leq i\leq n$ and $j\in J$, where the additional superscript $j$ indicates the additional use of a subordinate partition of unity.  In this paper we set external forces to zero, although there is no problem to include them into the scheme we proposed. This is just for the sake of formal simplicity of the description.
An equivalent assumption on the initial data ${\mathbf h}$ is they are located in Sobolev spaces  of arbitrary order $s\in {\mathbb R}$, i.e., for all $s\in {\mathbb R}$ we have
\begin{equation}\label{hHs}
{\mathbf h}\in  H^s\left( M, TM\right).
\end{equation}
As indicated in the introduction our proof of a bounded regular solution of the incompressible Navier-Stokes equation consists of three main ideas: a) we introduce a time discretization and a series of linear time transformations $t=\rho \tau$ such that  time step size $1$ in $\tau$-coordinates is related to a small time step size in original coordinates and small coefficients of spatial derivatives in transformed time coordinates. Then a local solution is constructed via two iterations. In a local iteration we determine a local Leray projection term by solving a local Poisson equation the boundary data imposed by the result of other local equations at the previous time step (where at the first iteration step the final data of the previous time step or the initial data at the first time step may be used in order to initialize the boundary conditions for the Poisson equation). For fixed boundary conditions we consider a subiteration in order to solve the local (local in time and local in space) Navier Stokes equations. The local Navier Stokes equation communicate via the boundary conditions of the Poisson equation which determines the Leray projection term. Such an double iteration procedure leads to a spatially global solution which is local in time. The choice of a constant step size $\rho >0$ depends on the size of the data, the manifold $M$, the viscosity information and drift information which is coded in the first and second order coefficients of the local equations. 
 b) in order to control the growth of the solution we introduce time-step by time step a control  function $\mathbf{r}$. Having determined the controlled velocity function $v^{r,\rho,l-1}(l-1,.)$ at the previous time step, we solve first locally in time, i.e., on the domain $[l-1,l]\times M$ the usual uncontrolled Navier Stokes equation, but with data $\mathbf{v}^{r,\rho,l-1}(l-1,.)$. The solution is denoted by $\mathbf{v}^{r^{l-1},\rho,l}$. Then a control function increment is chosen which depends only on the data $\mathbf{v}^{r,\rho,l-1}(l-1,.)$ and $\mathbf{r}^{l-1}(l-1,.)$ of the previous time step. It is defined locally in terms of 'consumption' source term $\phi^{v,l,ij}$ or $\phi^{r,l,ij}$ and  which has been explained in the introduction to some extent. These functions are chosen such that the growth of the functions $v^{r^{l-1},\rho,l,ij}$ over time step $l$, i.e., some norm of the increments $\delta v^{r^{l-1},\rho,l,ij}(l,.)$ is dominated by a respective norm of the increments $\delta r^{l,ij}$ of the control functions. 
 c) We ensure that the control function $\mathbf{r}$ and the function $\mathbf{v}^r$ are globally H\"{o}lder continuous with respect to space and time and bounded or at least of linear growth. This implies that classical arguments lead to classical $C^{1,2}$-regularity of the velocity function $\mathbf{v}$, and hence of the pressure.

The main result of this paper is that for a class of uniformly scalar parabolic operators $L$ acting on the components of the vector field $\mathbf{v}$, and which includes the Hodge and Bochner Laplacian on manifolds,    
the Navier Stokes equation
\begin{equation}\label{navm}
\begin{array}{ll}
\frac{\partial \mathbf{v}}{\partial t}-\nu L\mathbf{v }+\nabla_{\mathbf{v}}\mathbf{v}=-\nabla_Mp,\\
\\
\div\mathbf{v}=0,\\
\\
\mathbf{v}(0,.)=\mathbf{h}\in C^{\infty}\left(M,TM \right) ,
\end{array}
\end{equation}
we have a global scheme which converges to a global classical solution to (\ref{navm}) in its Leray projection form. In the scheme we use a local representation of the Leray projection operator $P$ which is the orthogonal projection of $L^2\left(M,TM \right)$ onto the kernel of the divergence.

We prove
\begin{thm}\label{mainthm}
Given any dimension $n$  let $\mathbf{h}\in C^{\infty}\left(M,TM\right)$ (or, equivalently satisfy (\ref{hHs}) for any $s\in {\mathbb R}$). Then there is a global classical solution
\begin{equation}
\mathbf{v}\in  C^{1,2}\left( \left[0,\infty \right)\times M,TM\right)
\end{equation}
of the Navier-Stokes equation system (\ref{navm}).
\end{thm}

\begin{proof}
We consider the schemes described in i) and ii) above and do the proof in four steps.
\begin{itemize}
 \item[1)] In a first step we prove for each $j\in J$ the local convergence of the series
 \begin{equation}\label{series1}
 \left( v^{r^{l-1},\rho,l,ij,m,p}\right)_{p\in {\mathbb N}}
 \end{equation}
in $C^{1,2}\left(\left[l-1,l\right]\times U_j\right)$ for all $1\leq i\leq n$. The members of the series in (\ref{series1}) are solutions of local uncontrolled Navier Stokes equations with data
\begin{equation}
v^{r^{l-1},\rho,l,ij,m,p}(l-1,.):=v^{r,\rho,l-1,ij}(l-1,.),
\end{equation}
i.e. we use as initial data the final data of the previous time step which are independent of the iteration index $p$ (of the spatially local iteration) and the iteration index $m$ (of the spatially global iteration) of the time-local convergence. As we use the controlled data of the previous time step, but have no control function increment $\delta r^{l,ij}$ defined on $[l-1,l]\times U_j$ involved in the computation of $v^{r^{l-1},\rho,l,ij,m,p}$ we use the superscript $r^{l-1}$ in order to indicate the dependence of the functions on the control function $\mathbf{r}^{l-1}:[0,l-1]\times M\rightarrow TM$ and its difference to the controlled velocity functions $v^{r,\rho,l,ij}$ even in the limit. In this step of the proof and in step 2) of the proof we shall show that for all $(\tau,x)\in [l-1,l]\times U_j$ all $j\in J$ and all $1\leq i\leq n$ we have 
\begin{equation}
v^{r,\rho,l,ij}(\tau,x)=\lim_{m,p\uparrow \infty}v^{r^{l-1},\rho,l,ij,m,p}(\tau,x)+\delta r^{l,ij}(\tau,x).
\end{equation}
 \item[2)] In a second step we prove convergence of the time-local and spatially global iteration scheme,where the data $\mathbf{r}^{l-1}(l-1,.)$ and $\mathbf{v}^{r,\rho,l-1}(l-1,.)$ are in $ C^2(M,TM)$. We consider the simplest iteration scheme where the communication between the local equations for the controlled velocity functions $v^{r,\rho,l,ij,m}$ and $r^{l,ij}$ is exclusively realized via the boundary conditions of the Leray projection terms of the local equations. In our representations of local solutions these boundary terms are coded in Green's functions. In this second step st we shall show that for all $(\tau,x)\in [l-1,l]\times U_j$ all $j\in J$ and all $1\leq i\leq n$ we have 
\begin{equation}
v^{r,\rho,l,ij}(\tau,x)=\lim_{m\uparrow \infty}v^{r^{l-1},\rho,l,ij,m}(\tau,x)+\delta r^{l,ij}(\tau,x),
\end{equation}
where for $j_1\neq j_2$, all $x\in U_{j_1}\cap U_{j_2}\neq \oslash$ we have for all $\tau\in [l-1,l]$ that
\begin{equation}
v^{r,\rho,l,ij_1}(\tau,x)=v^{r,\rho,l,ij_2}(\tau,x).
\end{equation}
 
 \item[3)] In a third step we consider first the scheme of item i) in the preceding section and prove that we have a global uniform bound for the controlled velocity function and a linear upper bound for the control function (which suffices in order to prove global existence). Then we consider the refined the scheme define in ii) of section 2 and prove a sharper result that states the existence of an uniform upper bound of the control function and the controlled velocity function. We prove that for this scheme a certain upper bound $C>0$ is preserved in two time steps, i.e., there exists a constant $C>0$ such that for all $l\geq 1$, all $j\in J$ and all $1\leq i\leq n$ we have
 \begin{equation}
\sup_{x\in U_j} {\big |}v^{r,\rho,l-1,ij}(l-1,x){\big |}\leq C\rightarrow \sup_{x\in U_j} {\big |}v^{r,\rho,l+1,ij}(l+1,x){\big |}\leq C.
 \end{equation}
 Furthermore the upper bound can be chosen such that a similar implication holds for the control functions, i.e., there is  a constant $C>0$ such that for all $l\geq 1$, all $j\in J$ and all $1\leq i\leq n$ we have
 \begin{equation}
\sup_{x\in U_j} {\big |}r^{l,ij}(l,x){\big |}\leq C\rightarrow \sup_{x\in U_j} {\big |}r^{l+1,ij}(l+2,x){\big |}\leq C.
 \end{equation}
 Strengthening this result we show that we can find such an upper bound for multivariate derivatives as well, provided that we have local sufficient local regularity of the controlled velocity functions and of the control functions, i.e., if $v^{r,\rho,l}\in C^{m,2m}\left([l-1,l]\times,M,TM \right)$ for some $m\geq 1$ and for $l-1\geq 0$, then for all multiindices $\alpha$ with $|\alpha|\leq m$ we have a constant $C_m$ such that
 \begin{equation}
\sup_{x\in U_j} {\big |}D^{\alpha}_xv^{r,\rho,l-1,ij}(l-1,x){\big |}\leq C_m\rightarrow \sup_{x\in U_j} {\big |}D^{\alpha}_xv^{r,\rho,l+1,ij}(l+1,x){\big |}\leq C_m.
 \end{equation} 
Similarly, if $r^{l}\in C^{m,2m}\left([l-1,l]\times,M,TM \right)$ for some $m\geq 1$ and for $l-1\geq 0$, then for all multiindices $\alpha$ with $|\alpha|\leq m$ we have a constant $C_m$ such that
\begin{equation}
\sup_{x\in U_j} {\big |}D^{\alpha}_xr^{l,ij}(l,x){\big |}\leq C_m\rightarrow \sup_{x\in U_j} {\big |}D^{\alpha}_xr^{l+1,ij}(l+2,x){\big |}\leq C_m.
 \end{equation}
 Finally we show that for the simplified scheme described in item i) of the introduction we still have a global upper bound. 
\item[4)] In a fourth step we show the existence of a globally bounded classical solution $\mathbf{v}^{\rho}\in C^{1,2}\left([0,\infty)\times M,TM \right)$ of the time-transformed incompressible Navier Stokes equation exists. It follows immediately that the global solution $\mathbf{v}\in C^{1,2}\left([0,\infty)\times M,TM \right)$ in original time coordinates exists. 
\end{itemize}

\subsection{step 1: proof of local existence of solutions at each time step (in absence of a control function) } 
We emphasize that the time step size $\rho$ is chosen generically, i.e. in this first step of the proof we shall choose $\rho >0$ such that local convergence is obtained in the scheme without control function. Similarly the bound $C>0$ of for the (modulus of) controlled velocity functions $v^{r,\rho,l,ij}$, and the (modulus of) the control functions $r^{l,ij}$ itself. Similarly for spatial derivatives of these functions.
For $l=1$ we set $r^{1,ij}(0,.)\equiv -\frac{h^{ij}}{C}$ for all $1\leq i\leq n$ and $j\in J$. As we said we may also define the control functions to be zero at the first time step, but in order to have a more uniform description for all time steps $l>0$ we use the prescription in terms of functions which are proportional to the negative data functions. In general we assume that the final data of the previous time step at $\tau=l-1$ are the initial data of time step $l$ and that the data for the controlled velocity functions at time step $l$, i.e., the functions
\begin{equation}
v^{r,\rho,l,ij}(l-1,.),~1\leq i\leq n,~j\in J,
\end{equation}
and the data of the control functions, i.e., the data
\begin{equation}
r^{l,ij}(l-1,.),~1\leq i\leq n,~j\in J
\end{equation}
are determined. Next we fix the iteration step $m\geq 1$ of the spatially global iteration. This determines that boundary data, which we use in order to determine the local Leray projection term using Green's function. For $m=1$ we take $v^{r^{l-1},\rho,l,ij,m-1}(\tau,x)=v^{r^{l-1},\rho,l-1,ij}(l-1,x)$ in order to determine the boundary conditions of the Poisson equation which determine the Leray projection term (as described in the introduction). For $m>1$ we assume that the functions
\begin{equation}
v^{r^{l-1},\rho,l,ij,m-1},~1\leq i\leq n,~j\in J,
\end{equation}
and the functions
\begin{equation}
r^{l,ij},~1\leq i\leq n,~j\in J
\end{equation}
have been determined. Recall the notation here: the upper script $r^{l-1}$ means that we solve local uncontrolled Navier Stokes equations with controlled data $v^{r^{l-1},\rho,l,ij}(l-1,.)=v^{r,\rho,l-1,ij}(l-1,.)$. As the control function increments $\delta r^{l,ij}=r^{l,ij}-r^{l-1,ij}$ are defined in terms of the data $v^{r,\rho,l-1,ij}(l-1,.)$ and $r^{l-1,ij}(l-1,.)$, we can add this control function increment after computation of the controlled velocity functions $v^{r^{l-1},\rho,l,ij},~1\leq i\leq n,j\in J$ in order to determine  $v^{r,\rho,l,ij}=v^{r^{l-1},\rho,l,ij}+\delta r^{l,ij},~1\leq i\leq n,j\in J$.

Next we fix $j\in J$, the iteration index $m\geq 1$ and choose flat coordinates on $U_j$. For $1\leq i\leq n$, $l\geq 1$ and the local iteration index $p\geq 1$ we solve the local Navier Stokes equation problem   for $v^{r,\rho,l,ij,m,p}$ of the form
\begin{equation}\label{eqiloc1mkproof}
\begin{array}{ll}
\frac{\partial v^{r^{l-1},\rho,l,ij,m,p}}{\partial \tau}-\rho\sum_{q,k=1}^na^{f,U}_{qk}\frac{\partial^2 v^{r^{l-1},\rho,l,ij,m,p}}{\partial x_q\partial x_k}\\
\\-\rho\sum_{k=1}^nb^{f,U}_k\frac{\partial v^{r,\rho,l,ij,m,p}}{\partial x_k}
+\rho\sum_{k=1}^nv^{r^{l-1},\rho,l,ij,m,p-1}_{,k}v^{r^{l-1},\rho,l,kj,m,p-1}\\
\\
=\rho S^{j}_{\mbox{int},i}\left( \mathbf{ v}^{r^{l-1},l,ij,m,p-1},\nabla \mathbf{v}^{r,l,ij,m,p-1}\right)\\
\\
+
\rho S^{J_j}_{\mbox{coup}}\left( \mathbf{ v}^{r^{l-1},l,ij,m-1},\nabla \mathbf{v}^{r^{l-1},l,ij,m-1}\right),\\
\\
v^{r^{l-1},\rho,l,ij,m,p}(l-1,.)=v^{r,\rho,l,ij}(l-1,.).
\end{array}
\end{equation}
Note that the coupling term is independent of the iteration index $p$. Furthermore the 'Burger equation term' 
\begin{equation}
 \rho\sum_{k=1}^nv^{r^{l-1},\rho,l,ij,m,p-1}_{,k}v^{r^{l-1},\rho,l,kj,m,p-1}
\end{equation}
is taken from the previous iteration step. This has the advantage that we can represent solutions in the form of convolutions with certain fundamental solutions of parabolic equations with constant coefficients as we use locally flat coordinates on $U_j$.

\begin{rem}
We may define additional boundary conditions at each iteration step $p$. Natural boundary conditions are of the form
\begin{equation}
v^{r^{l-1},\rho,l,ij,m,p}(l-1,.)=v^{r,\rho,l-1,ij}(l-1,.),
\end{equation}
and for each $j\in J$ we added a boundary condition 
\begin{equation}\label{boundstandproofb}
v^{r^{l-1},\rho,l,ij,m,p}|_{[l-1,l]\times \partial U_j}(\tau,x)=\sum_{k\in J_j}v^{r^{l-1},\rho,l,i,kk,m-1}(\tau,x).
\end{equation}
Note the double superscript $kk$ which indicates the use of a partition of unity. This ensures that differentiability of the restriction of $v^{r,\rho,l,ij,m,p}$ to $[l-1,l]\times \partial U_j$ in (\ref{boundstandproofb}).
In any case we have no dependence on the boundary condition if we consider iteration with respect to the iteration index $p$.
\end{rem}
The solution of the local uncontrolled system with controlled data $v^{r^{l-1},\rho,l,ij,m}$ (limit $p\uparrow \infty$) may be represented in the form
\begin{equation}
v^{r^{l-1},\rho,l,ij,m}=v^{r^{l-1},\rho,ij,m,1}+\sum_{p=2}^{\infty}\delta v^{r^{l-1},\rho,l,ij,m,p}
\end{equation}
along with $\delta v^{r^{l-1},\rho,l,ij,m,p}=v^{r^{l-1},\rho,l,ij,m,p}-v^{r^{l-1},\rho,l,ij,m,p-1}$. For $p=1$ we denote 
$\delta v^{r^{l-1},\rho,l,ij,m,1}=v^{r^{l-1},\rho,l,ij,m,1}-v^{r^{l-1},\rho,l,ij,m,0}=v^{r^{l-1},\rho,l,ij,m,1}-v^{r^{l-1},\rho,l-1,ij,m}(l-1,.)$. This has the advantage of zero initial conditions and zero boundary conditions for $\delta v^{r^{l-1},\rho,l,ij,m,p}$ for $p\geq 2$. Furthermore the equation for $v^{r^{l-1},\rho,l,ij,m,1}$ is a linear parabolic equation with differentiable coefficients. The equation for $\delta v^{r^{l-1},\rho,l,ij,m,p}$ becomes
\begin{equation}\label{eqiloc1mkproofdelta}
\begin{array}{ll}
\frac{\partial \delta v^{r^{l-1},\rho,l,ij,m,p}}{\partial \tau}-\rho\sum_{q,k=1}^na^{f,U}_{qk}\frac{\partial^2 \delta v^{r^{l-1},\rho,l,ij,m,p}}{\partial x_q\partial x_k}-\rho\sum_{k=1}^nb^{f,U}_k\frac{\partial \delta v^{r^{l-1},\rho,l,ij,m,p}}{\partial x_k}\\
\\
-\rho\sum_{k=1}^n\delta v^{r^{l-1},\rho,l,ij,m,p-1}_{,k}v^{r^{l-1},\rho,l,kj,m,p-1}-\rho\sum_{k=1}^n\delta v^{r^{l-1},\rho,l,ij,m,p-1}_{,k}v^{r^{l-1},\rho,l,kj,m,p}\\
\\
=-\rho\sum_{k=1}^n v^{r^{l-1},\rho,l,ij,m,p-1}_{,k}\delta v^{r^{l-1},\rho,l,kj,m,p-1}-\rho\sum_{k=1}^n \delta v^{r^{l-1},\rho,l,ij,m,p-1}_{,k} v^{r^{l-1},\rho,l,kj,m,p-1}\\
\\
+\rho S^{j}_{\mbox{int},i}\left( \mathbf{ v}^{r^{l-1},l,ij,m,p-1},\nabla \mathbf{v}^{r^{l-1},l,ij,m,p-1}\right)\\
\\
-\rho S^{j}_{\mbox{int},i}\left( \mathbf{ v}^{r^{l-1},l,ij,m,p-2},\nabla \mathbf{v}^{r^{l-1},l,ij,m,p-2}\right),
\end{array}
\end{equation}
with zero initial and boundary conditions which do not depend on the iteration index $p$.
Note that the coupling term $\rho S^{J_j}_{\mbox{coup}}\left( \mathbf{ v}^{r^{l-1},l,ij,m-1},\nabla \mathbf{v}^{r^{l-1},l,ij,m-1}\right)$ disappears since it does not depend on the iteration index $p$.
Hence the solution for the functions $\delta v^{r^{l-1},\rho,l,ij,m,p}$ for $p\geq 1$ has the representation

\begin{equation}
\delta v^{r^{l-1},\rho,l,ij,m,p}(\tau,x)=\int_{l-1}^{\tau}\int_{U_j}\delta s^{r^{l-1},\rho,l,ij,m,p-1}_{\mbox{int}}(s,y)p^{l,ij}(\tau,x;s,y)dyds,
\end{equation}
where $p^{l,ij}$ is the fundamental solution of
\begin{equation}\label{eqiloc1mkprooffund}
\begin{array}{ll}
\frac{\partial v^{r^{l-1},\rho,l,ij,m,p}}{\partial \tau}-\rho\sum_{q,k=1}^na^{f,U}_{qk}\frac{\partial^2 v^{r^{l-1},\rho,l,ij,m,p}}{\partial x_q\partial x_k}-\rho\sum_{k=1}^nb^{f,U}_k\frac{\partial v^{r^{l-1},\rho,l,ij,m,p}}{\partial x_k}
=0,
\end{array}
\end{equation}
and where $\delta s^{r^{l-1},\rho,l,ij,m,p-1}_{\mbox{int}}(s,y)$ is an abbreviation for the right side of (\ref{eqiloc1mkproofdelta}). Note that, the equation in (\ref{eqiloc1mkprooffund}) is a linearly transformed heat equation as we use locally flat coordinates .

From classical theory of scalar parabolic equations we have for $p\geq 1$ that
\begin{equation}
v^{r^{l-1},\rho,l,ij,m,1},~v^{r^{l-1},\rho,l,ij,m,p}\in C^{1,2}\left([l-1,l]\times \overline{U_j}\right). 
\end{equation}

\begin{rem}
Even if we consider a scheme with variable first order terms, or a generalised scheme with variable second order and first order coefficients, then we can use constructions of  the fundamental solution $p^{l,ij,m,p}$ in terms of the Levy expansion. Recall that on an arbitrary domain $[l-1,l]\times \Omega$ form is given by
\begin{equation}
p^{l,ij}(\tau,x;s,y):=N^l_A(\tau,x;s,y)+\int_s^{\tau}\int_{{\mathbb R}^n}N^l_A(\tau,x;\sigma,\xi)\phi(\sigma,\xi;s,y)d\sigma d\xi,
\end{equation}
where for $\left( a^{ij,U_j}(y)\right)$ defining the inverse of $\left( a^{U_j}_{ij}(y)\right)$ we have
\begin{equation}
N^l_A(\tau,x;s,y)=\frac{\sqrt{\det\left[ a^{ij,U_j}(y)\right] }}{\sqrt{\sum_{i,j=1}^n4\pi \rho (\tau-s)}^n}\exp\left(-\frac{\sum_{i,j=1}^n
a^{ij,U_j}(x^i-y^i)(x^j-y^j)}{4\rho_l\nu (\tau-s)} \right),
\end{equation}
and $\phi$ is a recursively defined function which is H\"{o}lder continuous in $x$, i.e.,
\begin{equation}
\phi(\tau,x;s,y)=\sum_{m=1}^{\infty}(L_lN^l_A)_m(\tau,x;s,y),
\end{equation}
along with the recursion
\begin{equation}
\begin{array}{ll}
(L_lN^l_A)_1(\tau,x;s,y)=L_lN^l_A(\tau,x;s,y)\\
\\
=\frac{\partial N^l_A}{\partial \tau}-\rho a_{ij}^{l,U_j} \frac{\partial^2}{\partial x_i\partial x_j}N^l_{A}+\rho\sum_{k=1}^n b^{U_j}_j\frac{\partial N^l_A}{\partial x_k}\\
\\
=\rho\sum_{k=1}^n b_k\frac{\partial N^l_A}{\partial x_k},\\
\\
(LN^l_A)_{m+1}(\tau ,x):=\int_s^t\int_{\Omega}\left( LN^l_A(\tau,x;\sigma,\xi)\right)_m LN^l_A(\sigma,\xi;s,y)d\sigma d\xi.
\end{array}
\end{equation}
We may then use the adjoint of the fundamental solution in order to obtain estimates similar as the estimates below which we shall get for the simplified convolutive expressions of approximating local solution functions $v^{r^{l-1},\rho,l,ij,m,p}$. 
Note that for small $\rho>0$ the Levy expansion is a kind of perturbation of the leading term $N^l_A$.
\end{rem}
 For locally flat coordinates we have a fundamental solution $p^{l,ij,m,p-1}$ which depends on the time difference $\tau-s$ and on the spatial differences $x-y$ such that there is a function $p^{*,l,ij}$ such that
 \begin{equation}
 p^{*,l,ij}(\tau-s,x-y)=p^{l,ij}(\tau,x;s,y).
 \end{equation}
At this point we shall see that it becomes advantageous if we have imposed boundary conditions as in (\ref{boundstandproofb}). Then the increments have zero boundary conditions and we may apply partial integration where boundary terms disappear. 
 Especially for second partial derivatives with respect to the variables $x_k$ and $x_q$ we have
\begin{equation}\label{solstep0rep}
\begin{array}{ll}
\delta v^{r^{l-1},\rho,l,ij,m,p}_{,k,q}(\tau,x)=\int_{l-1}^{\tau}\int_{U_j}\delta s^{r^{l-1},\rho,l,ij,m,p-1}_{\mbox{int}}(s,y)p^{l,ij}_{,k,q}(\tau,x;s,y)dyds\\
\\
=\delta v^{r^{l-1},\rho,l,ij,m,p}_{,k,q}(\tau,x)=\int_{l-1}^{\tau}\int_{U_j}\delta s^{r^{l-1},\rho,l,ij,m,p-1}_{\mbox{int},q}(s,y)p^{l,ij}_{,k}(\tau,x;s,y)dyds,
\end{array}
\end{equation}
and then we use local integrability of the first order derivatives of the transformed Gaussian and other properties of the Gaussian in order to estimate this representation.
For some constant $c>0$ we get
\begin{equation}
{\big |}\delta s^{r^{l-1},\rho,l,ij,m,p-1}_{\mbox{int},q}(s,y){\big |}\leq \rho c\sum_{0\leq |\alpha|\leq 2}\sup_{\tau\in[l-1,l],x\in U_j}{\big |}\delta v^{r^{l-1},\rho,l,ij,m,p-1}_{,\alpha}(\tau,x){\big |},
\end{equation}
and since the Gaussian $p^{l,ij}$ and its first spatial derivatives $p^{l,ij}_{,k}$ are locally integrable, i.e., have the upper bounds
\begin{equation}
{\big |}p^{l,ij}(\tau-s,x-y){\big |}\leq \frac{C}{(\tau-s)^{\sigma}(x-y)^{n-2\sigma}}
\end{equation}
\begin{equation}
{\big |}p^{l,ij}_{,k}(\tau-s,x-y){\big |}\leq \frac{C}{(\tau-s)^{\mu}(x-y)^{n+1-2\sigma}}
\end{equation}
for some $C>0$ and $\sigma\in (0.5,1)$, we get from (\ref{solstep0rep})
\begin{equation}\label{solstep0rep}
\begin{array}{ll}
\sup_{\tau\in[l-1,l],x\in U_j}{\big |}\delta v^{r^{l-1},\rho,l,ij,m,p}_{,k,q}(\tau,x){\big |}\\
\\
\leq \rho C\sum_{0\leq |\alpha|\leq 2}\sup_{\tau\in[l-1,l],x\in U_j}{\big |}\delta v^{r^{l-1},\rho,l,ij,m,p-1}_{,\alpha}(\tau,x){\big |}
\end{array}
\end{equation}
We get similar estimates for the first order derivatives and for the value function itself, i.e., as $n\geq 2$, for $1+n+n^2< 2n^2 $ terms and a generic constant $C>0$.  Hence for $\rho\leq \frac{1}{4n^2C}$ we get the contraction 
\begin{equation}
|\delta v^{r^{l-1},\rho,l,ij,m,p}|_{1,2}\leq \frac{1}{4} |\delta v^{r^{l-1},\rho,l,ij,m,p-1}|_{1,2}.
\end{equation} 
Furthermore, we may assume that $\rho>0$ is small enough such that
\begin{equation}
|\delta v^{r^{l-1},\rho,l,ij,m,1}|_{1,2}\leq \frac{1}{4}.
\end{equation}
For this time-step size $\rho>0$ we have
\begin{equation}
{\Big |}\sum_{p=2}^{\infty}\delta v^{r^{l-1},\rho,l,ij,m,p}{\Big |}_{1,2}\leq \frac{1}{2},
\end{equation}
(strictly less indeed).
Hence for this time-step size $\rho>0$ the sequence
\begin{equation}
\left( v^{r^{l-1},\rho,l,ij,m,p}\right)_{p\in {\mathbb N}}
\end{equation}
converge classically to a fixed point limit
\begin{equation}
v^{r^{l-1},\rho,l,ij,m}:=\lim_{p\uparrow \infty}v^{r,\rho,l,ij,m,p}\in C^{1,2}\left([0,1]\times U_j\right) 
\end{equation}
for all $1\leq i\leq n$ and $j\in J$, and in Banach space $C^{1,2}\left([l-1,l]\times \overline{U_j} \right)$. For each $m\geq 1$ this fixed point limit solves the initial-boundary value problem for $v^{r^{l-1},\rho,l,ij,m}$ stated in (\ref{eqiloc1m1stepmproof}), (\ref{initialm}), and (\ref{boundstandproof}) below.

\subsection{step 2: Convergence of the spatially global and time-local controlled scheme} 
 
Again, we could define equations for functions $v^{r^{l-1},\rho,l,ij,m}$ where data of nonlinear terms are taken from the previous iteration step in order to get convolutions in terms of Gaussians in locally flat coordinates. However, an alternative is the following:  
 since we {\it know} the functions $v^{r^{l-1},\rho,l,ij,m}$ as the fixed point limits of subiteration steps with iteration index $p$ for all $1\leq i\leq n$ and all $j\in J$, we can represent the functions $v^{r,\rho,l,ij,m}$ in terms of fundamental solutions $p^{l,ij,m}$ of the  equation  
\begin{equation}
\begin{array}{ll}
\frac{\partial p^{l,ij,m}}{\partial \tau}-\rho\sum_{q,k=1}^na^{f,U}_{qk}(x)\frac{\partial^2 p^{l,ij,m}}{\partial x_q\partial x_k}-\rho\sum_{k=1}^nb^{f,U}_k(x)\frac{\partial p^{l,ij,m}}{\partial x_k}\\
\\
+\rho\sum_{k=1}^np^{l,ij,m}v^{r^{l-1},\rho,l,kj,m}
=0.
\end{array}
\end{equation}
We may then use estimates similar as in the previous section, where we can use the adjoint of the fundamental solution.
The functions function $v^{r^{l-1},\rho,l,ij,m}$ are elements of a functional sequence $\left( v^{r^{l-1},\rho,1,ij,m}\right)_{m\in {\mathbb N}}$, where we want to show that the limit
\begin{equation}
v^{r^{l-1},\rho,l,ij}:=\lim_{m\uparrow \infty}v^{r^{l-1},\rho,l,ij,m},~1\leq i\leq n~j\in J
\end{equation}
is a local representation of a time-local and spatially global function
\begin{equation}
\mathbf{v}^{r^{l-1},\rho,l}:\left[l-1,l \right]\times M\rightarrow TM, 
\end{equation}
which solves the incompressible Navier Stokes equation on manifolds on the domain $[l-1,l]\times M$, provided that data satisfy 
\begin{equation}
\mathbf{v}^{r^{l-1},\rho,l}(l-1,.)\in C^{1,2}\left(  M\right).
\end{equation}
We have a spatially global iteration scheme of local initial-boundary value problems of the form
\begin{equation}\label{eqiloc1m1stepmproof}
\begin{array}{ll}
\frac{\partial v^{r^{l-1},\rho,l,ij,m}}{\partial \tau}-\rho\sum_{q,k=1}^na^{f,U}_{qk}(x)\frac{\partial^2 v^{r^{l-1},\rho,l,ij,m}}{\partial x_q\partial x_k}-\rho\sum_{k=1}^nb^{f,U}_k(x)\frac{\partial v^{r^{l-1},\rho,l,ij,m}}{\partial x_k}\\
\\
+\rho\sum_{k=1}^nv^{r^{l-1},\rho,l,ij,m}_{,k}v^{r^{l-1},\rho,l,kj,m}\\
\\
=\rho S^j_{\mbox{int},i}\left( \mathbf{ v}^{r^{l-1},\rho,l,m},\nabla \mathbf{v}^{r^{l-1},\rho,l,m}\right)+\rho S^{j}_{\mbox{coup}}\left(\mathbf{ v}^{r^{l-1},\rho,l,m-1},\nabla \mathbf{v}^{r^{l-1},\rho,l,m-1}\right),
\end{array}
\end{equation}
where for $m=1$ we set $v^{r^{l-1},\rho,l,kj,m-1}=v^{r^{l-1},\rho,l,kj,0}:=v^{r^{l-1},\rho,l-1,kj}(l-1,.)$.
At each iteration step $m$ we defined
\begin{equation}\label{initialm}
v^{r^{l-1},\rho,l,ij,m}(l-1,.)=v^{r^{l-1},\rho,l-1,ij}(l-1,.),
\end{equation}
and for each $j\in J$ we added a boundary condition 
(for $(\tau,x)\in [l-1,l]\times \partial U_j$)
\begin{equation}\label{boundstandproof}
v^{r^{l-1},\rho,l,ij,m}|_{[l-1,l]\times \partial U_j}(\tau,x)=\sum_{k\in J_j}v^{r^{l-1},\rho,l,ikk,m-1}(\tau,x).
\end{equation}
For all $(\tau,x)\in \overline{U_j}$ we have the representation
\begin{equation}\label{repstep1}
\begin{array}{ll}
v^{r^{l-1},\rho,l,ij,m}(\tau,x)=\int_{U_j} v^{r^{l-1},\rho,l-1,ij}(l-1,y)p^{l,ij,m}(\tau,x;0,y)dy\\
\\
+\int_{l-1}^{\tau}\int_{U_j}s^{r^{l-1},\rho,l,ij,m}_{\mbox{int},\mbox{coup}}(s,y)p^{l,ij,m}(\tau,x;s,y)dsdy\\
\\
+\int_{l-1}^{\tau}\int_{\partial U_j}\phi_{\mbox{bd}}(s,y)p^{l,ij,m}(\tau,x;s,y)dS_yds
\end{array}
\end{equation}
where we used the abbreviation
\begin{equation}
\begin{array}{ll}
s^{r^{l-1},\rho,l,ij,m}_{\mbox{int},\mbox{coup}}:=\rho S^{j}_{\mbox{int},i}\left(  v^{r^{l-1},l,ij,m},\nabla \mathbf{v}^{r^{l-1},l,ij,m}\right)+\\
\\
\rho S^{J_j}_{\mbox{coup}}\left( \mathbf{ v}^{r^{l-1},l,ij,m-1},\nabla \mathbf{v}^{r^{l-1},l,ij,m-1}\right),
\end{array}
\end{equation}
and where $dS_y$ denotes a surface element on $\partial U_j$. The boundary relation reduces to an integral equation
\begin{equation}\label{repstep2}
\begin{array}{ll}
\int_{U_j} v^{r^{l-1},\rho,l-1,ij}(l-1,y)p^{l,ij,m}(\tau,x;0,y)dy\\
\\
+\int_{l-1}^{\tau}\int_{U_j}s^{r^{l-1},\rho,l,ij,m}_{\mbox{int},\mbox{coup}}(s,y)p^{l,ij,m}(\tau,x;s,y)dsdy\\
\\
+\int_{l-1}^{\tau}\int_{\partial U_j}\phi_{\mbox{bd}}(s,y)p^{l,ij,m}(\tau,x;s,y)dS_yds\\
\\
=\sum_{k\in J_j}v^{r^{l-1},\rho,l,ikk,m-1}(\tau,x)
\end{array}
\end{equation}
for the function $\phi_{bd}$.
Again we may solve for $n$ scalar initial boundary value problems for $v^{r^{l-1},l,ij,1}\in C^{1,2}\left([l-1,l]\times U_j\right)$ first, show that for small $\rho>0$ we have a small difference $v^{r^{l-1},l,ij,1}-v^{r^{l-1},l-1,ij}(l-1,.)$ and then show that $\delta v^{r^{l-1},l,ij,m}=v^{r^{l-1},l,ij,m}-v^{r,l,ij,m-1}$ satisfies a contraction
\begin{equation}\label{firstmstep}
|\delta v^{r^{l-1},l,ij,m}|_{1,2}\leq \frac{1}{4}|\delta v^{r^{l-1},l,ij,m-1}|_{1,2}.
\end{equation}
This is done using the classical representations of initial-boundary value problems in terms of fundamental solutions as above in (\ref{repstep1}) and (\ref{repstep2}).  Note that we use the term for small $\rho>0$ in a generic sense here, i.e., we first determine a $\rho$ such that we get the desired contraction for $\delta v^{r^{l-1},l,ij,m,p}$ with respect to the subiteration index $p$, and then we use this $\rho$ in order to get another $\rho$ which is smaller or equal such that the (\ref{firstmstep}) is satisfied. We do this here for one time step $l$, and the choice of the control function willendure that it can be done independently of the time step number $l$. 
For a step size $\rho>0$ which is small enough the sequences
\begin{equation}
\left( v^{r^{l-1},\rho,l,ij,m}\right)_{m\in {\mathbb N}}
\end{equation}
converge to a classical limit
\begin{equation}
v^{r^{l-1},\rho,l,ij}:=\lim_{m\uparrow \infty}v^{r^{l-1},\rho,l,ij,m}\in C^{1,2}\left([0,1]\times U_j\right) 
\end{equation}
for all $1\leq i\leq n$ and $j\in J$.
Finally we set
\begin{equation}
v^{r^{l-1},\rho,l,ij}:=v^{r,\rho,l-1,ij}(l-1,.)
\end{equation}
for all $1\leq i\leq n$ and $j\in J$. We may choose $C>0$ such that 
\begin{equation}
|v^{r^{l-1},\rho,l,ij}|_{1,2}\leq C.
\end{equation}

\subsection{step 3: Control of the growth of the functions $\mathbf{r}^l$ and $\mathbf{v}^{r,\rho,l}$ }
Before we analyze global upper bounds in time let is make a remark concerning the time step size. The local contraction result explained in step 1 and step 2 of this proof shows that the increment of the locally uncontrolled velocity function with controlled data $\mathbf{v}^{r^{l-1},\rho,l}\in C^{1,2}\left([l-1,l]\times M\right)$, i.e. the increment
\begin{equation}\label{incrementdelta3}
\delta \mathbf{v}^{r^{l-1},\rho,l}=\mathbf{v}^{r^{l-1},\rho,l}-\mathbf{v}^{r,\rho,l-1}(l-1,.),
\end{equation}
has an upper bound which decreases with the time step size $\rho$ (which appears in the symbol of the local operator via time transformation). Similar for all spatial derivatives as long as local regularity ensures that they are itself of some regularity (at least continuous). We define the upper bound via a local representation of (\ref{incrementdelta3}). For an $\epsilon >0$ depending on the upper bound $C>0$ of the data at time $l-1$, modell parameters such as viscosity or diffusion constants, and structural information of the manifold (including dimension) we can realize a bound 
\begin{equation}
\max_{1\leq i\leq n,j\in J}\sum_{0\leq |\alpha|\leq 2}\sup_{(\tau,x)\in [l-1,l]\times U_j}{\big |}D^{\alpha}_x\delta v^{r,\rho,l-1,ij}(\tau,x){\big |}\leq \epsilon
\end{equation}
which becomes small with the time step size $\rho>0$. As our local iteration scheme starts with the data $v^{r,\rho,l-1,ij}(l-1,.),~1\leq i\leq n,j\in J$ at time step $l\geq 1$ and leads to local iteration schemes of linear coupled parabolic equations with bounded coefficients for the simplest scheme in item i) it suffices to choose a small but constant time step size to preserve the upper bound - although the control function is allowed to have linear growth. In the analysis of the more involved scheme of item ii) of section 2 we shall have a uniform global upper bound for the control function and the controlled velocity function such that a constant time step size $\rho>0$ can be chosen anyway. However, if we consider a scheme with explicit equations for the controlled velocity function  which include the control function (as in item iii) of section 2), and if we consider a simple scheme, then we should better use a decreasing time step size $\rho_l$, i.e., the choice 
\begin{equation}\label{choicerho}
\mbox{simple scheme}+\mbox{item }iii)\Rightarrow \rho_l\sim \frac{1}{l}
\end{equation}
keeps the coefficients of the more involved local iteration equations uniformly in this case too, and this is certainly an advantage, while the choice in (\ref{choicerho}) leads still to a global scheme.

Next, we first prove that the scheme defined in item i) of the second section of this paper is global, i.e., that the controlled velocity functions are uniformly bounded and that the control functions have a global linear upper bound. The result is then sharpenend when we consider the extended control functions of item ii) of the second section of this paper in the sense that we get a global uniform upper bound for the control functions and of the controlled velocity functions.  
In the previous step of this proof we have obtained a local solution
\begin{equation}
\mathbf{v}^{r^{l-1},\rho,l}\in C^{1,2}\left( [l-1,l]\times M, TM\right),
\end{equation}
represented by a finite family of local functions $v^{r^{l-1},\rho,l,ij},~1\leq i \leq n,~j\in J$ via charts with image $U_j,~j\in J$ -provided that the initial data $\mathbf{v}^{r^{l-1},\rho,l}(l-1,.)=\mathbf{v}^{r,\rho,l-1}(l-1,.)$ are well-defined in $C^{2}\left(M,TM \right)$. At time step $l-1$ the control function $r^{l-1,ij},~1\leq i\leq n,~j\in J$ are known in addition. The control functions $r^{l,ij},~1\leq i\leq n,~j\in J$ at time step $l\geq 1$ are then defined for all $(\tau,x)\in [l-1,l]\times U_j$  by
\begin{equation}
r^{l,ij}(\tau,x)=r^{l-1,ij}(l-1,x)+\delta r^{l,ij}(\tau,x)
\end{equation}
where we want to choose the control functions increments $\delta r^{l,ij},~1\leq i\leq n,~j\in J$ such that the growth is controlled. Since we have local regular solutions, for this purpose of proving boundedness it is sufficient that the controlled velocity functions and the control function have an upper bound $C>0$ which is preserved inductively after finitely many steps. Now for the simple scheme of item i) of section 2 it is indeed not difficult to observe that for a small time step size $\rho>0$ the upper bound $C$ is preserved for the controlled velocity function for each time step. This follows from the definition of the simplified control function increment in item i) of section 2. We have
\begin{equation}\label{deltarrr}
\begin{array}{ll}
\delta r^{l,ij}(l,x):=\int_{l-1}^{l}\int_{U_j}\left( -\frac{v^{r,\rho,l-1,ij}(l-1,y)}{C}\right)C_{U_j}p^{l,ijj}(\tau-(l-1),x-y)dy.
\end{array}
\end{equation}
We mentioned in section 2 that $C_{U_j}$ is a normalisation constant which normalizes the spatial integral of the density $p^{l,ijj}$ to $1$ and is optional.  
Now for small time step size $\rho>0$ the integral in (\ref{deltarrr}) is close to the value $-\frac{v^{r,\rho,l-1,ij}(l-1,x)}{C}$. Especially, we may choose the time step size $\rho>0$ such that
\begin{equation}\label{measureest}
{\big |}\delta r^{l,ij}(l,x){\big |}> \frac{1}{2}\mbox{ if }{\Big |}\frac{v^{r,\rho,l-1,ij}(l-1,x)}{C}{\Big |}\geq \frac{3}{4}.
\end{equation}
Note that we can keep the estimate (\ref{measureest}) as we consider appropriate partitions of unity, but this is clear such that may suppress the additional indices.
Furthermore the local contraction result show that for the modulus of the local increment $\delta v^{r^{l-1},\rho,l,ij}(l,x)$ we have
\begin{equation}
{\big |}\delta v^{r^{l-1},\rho,l,ij}(l,x){\big |}\leq \frac{1}{2}.
\end{equation}

As the modulus of a data value $v^{r,\rho,l-1,ij}(l-1,y)$ becomes  close to $C$, let's say 
\begin{equation}
{\big |}v^{r,\rho,l-1,ij}(l-1,x){\big |}\geq \frac{3}{4}C
\end{equation}
at time $l-1$ we get
\begin{equation}
\begin{array}{ll}
{\big |}v^{r,\rho,l,ij}(l,x){\big |}={\big |}v^{r,\rho,l-1,ij}(l-1,x)+\delta v^{r^{l-1},\rho,l,ij}(l,x)+\delta r^{l,ij}(l,x){\big |}\\
\\
\leq {\big |}v^{r,\rho,l-1,ij}(l-1,x){\big |}\leq C
\end{array}
\end{equation}
A similar reasoning holds for preservation of a upper bound $C$ for multivariate derivatives of order $m=2$ from time $l-1$ to time $l$ if
\begin{equation}\label{preservationprep}
\begin{array}{ll}
\max_{1\leq i\leq n,j\in J}\sum_{0\leq |\alpha|\leq 2}\sup_{x\in  U_j}{\big |}D^{\alpha}_xv^{r,\rho,l-1,ij}(l-1,x){\big |}\leq C
\end{array}
\end{equation}
holds, i.e., we have
\begin{equation}\label{preservation1v}
\begin{array}{ll}
\max_{1\leq i\leq n,j\in J}\sum_{0\leq |\alpha|\leq 2}\sup_{x U_j}{\big |}D^{\alpha}_xv^{r,\rho,l-1,ij}(l-1,x){\big |}\leq C\\
\\
\rightarrow \max_{1\leq i\leq n,j\in J}\sum_{0\leq |\alpha|\leq 2}\sup_{x\in  U_j}{\big |}D^{\alpha}_xv^{r,\rho,l+1,ij}(l,x){\big |}\leq C.
\end{array}
\end{equation}
As we have local existence and regularity results the observation of a preservation of an upper bound inductively from time step (\ref{preservation1v}) suffices in order to conclude later that have a global upper bound $C'\leq C+1$ for all time.
As we have the inductive upper bound $C$ it follows from (\ref{deltarrr}) that we have a linear upper bound for the control function
\begin{equation}\label{rdeltarrr}
\begin{array}{ll}
{\big |}r^{l,ij}(l,x){\big |}\leq l+1,
\end{array}
\end{equation}
and
\begin{equation}\label{rdeltarrr}
\begin{array}{ll}
{\big |}D^{\alpha}_xr^{l,ij}(l,x){\big |}\leq l+1~\mbox{for all }~|\alpha|\leq 2,
\end{array}
\end{equation}  
where we use the assumptions
\begin{equation}
\max_{1\leq i\leq n,j\in J}\sup_{x\in U_j}{\big |}D^{\alpha}_xr^{0,ij}(x){\big |}\leq \max_{1\leq i\leq n,j\in J}\sup_{x\in U_j}{\Big |}D^{\alpha}_x\frac{h^{ij}(x)}{C}{\Big |}\leq 1
\end{equation}
for all $|\alpha|\leq 2$.
Hence we have a global linear upper bound for the functions $v^{\rho,l,ij}=v^{r,\rho,l,ij}-r^{l,ij}$ (linear growth at most with respect to time $l$) and the scheme becomes global.

Next we sharpen this result a bit. 
We fix $j\in J$ and assume that we have the upper bounds
\begin{equation}\label{maxv}
\max_{1\leq i\leq n}\sup_{x\in U_j}{\big |}v^{r,\rho,l-1,ij}(l-1,x){\big |}\leq C,
\end{equation}
and
\begin{equation}\label{maxr}
\max_{1\leq i\leq n}\sup_{x\in U_j}{\big |}r^{l-1,ij}(l-1,x){\big |}\leq C.
\end{equation}
The natural extension is an introduction of a switch which realizes the following idea: keep on going with the simple scheme as long as the control functions $D^{\alpha}_xr^{l-1,ij}(l-1,.),~l\geq 1,~0\leq |\alpha|\leq 2$ have the upper bound $C>0$. However, if such an upper bound does not hold for some $1\leq i\leq n$ and $j\in J$ and some $\alpha$ with $0\leq |\alpha|\leq 2$ then switch to a different definition of a control function increment which ensures that the modulus of the control function decreases during the next time step.  
Let 
\begin{equation}
M^{l-1,\alpha}_r:=\max_{1\leq i\leq n,~j\in J}\sup_{x\in U_j}{\big |}D^{\alpha}_xr^{l,ij}(l-1,x){\big |},
\end{equation}
and consider a property $P$ of the form
\begin{equation}\label{Pcond}
\begin{array}{ll}
\mbox{P}:~M^{l-1,\alpha}_r\leq C~\mbox{or for all $\alpha$ with}~0\leq |\alpha|\leq 2
\end{array}
\end{equation}
Then we simply write $P$ if the condition $\mbox{P}$ in (\ref{Pcond}) holds and $\mbox{non-P}$ if the condition $\mbox{P}$ in (\ref{Pcond}) does not hold. Now the definition of the control function increments in item ii) is 
\begin{equation}
\begin{array}{ll}
\delta r^{l,ij}(\tau,x):=\\
\\
\left\lbrace \begin{array}{ll}
				\int_{l-1}^{\tau}\int_{U_j}\left( -\frac{v^{r,\rho,l-1,ij}(l-1,y)}{C}\right)C_{U_j}p^{l,ijj}(\tau-(l-1),x-y)dy~\mbox{if $\mbox{P}$}\\
				\\
				\int_{l-1}^{\tau}\int_{U_j}\left( -\frac{r^{l-1,ij}(l-1,y)}{C}\right)C_{U_j}p^{l,ijj}(\tau-(l-1),x-y)dy~\mbox{if $\mbox{non-P}$}.
                                \end{array}\right.
\end{array}
\end{equation}
For the source terms involved we use the abbreviations
\begin{equation}
\phi^{v,l,ij}=-\frac{v^{r,\rho,l-1,ij}(l-1,y)}{C},
\end{equation}
and
\begin{equation}
\phi^{r,l,ij}=-\frac{r^{l-1,ij}(l-1,y)}{C}.
\end{equation} 

Now let us observe the growth behavior for two time steps. First assume that the property $P$ holds. For all $1\leq i\leq n$ and all $j\in J$ and $x\in U_j$ we have from time  $l-1$ to time $l$ the growth behavior
\begin{equation}\label{growthfirststep}
\begin{array}{ll}
\delta v^{r,\rho,l,ij}(l,x)=v^{r,\rho,l,ij}(l,x)-v^{r,\rho,l-1,ij}(l-1,x)\\
\\
=v^{r^{l-1},\rho,l,ij}(l,x)-v^{r,\rho,l-1,ij}(l-1,x)+\delta r^{l,ij}(l,x)\\
\\
=\delta v^{r^{l-1},\rho,l,ij}(l,x)+\delta r^{l,ij}(l,x)\\
\\
=\delta v^{r^{l-1},\rho,l,ij}(l,x)+\int_{l-1}^l\int_{U_j}\phi^{v,l,ij}(l-1,y)C_{U_j}p^{l,ijj}(1,x-y)dyds,
\end{array}
\end{equation}
where the latter integrand is independent of $s$, and the whole latter integral
\begin{equation}\label{sourceintv}
\int_{l-1}^l\int_{U_j}\phi^{v,l,ij}(l-1,y)C_{U_j}p^{l,ijj}(1,x-y)dyds
\end{equation}
is close to $\phi^{v,l,ij}(l-1,x)$ as the time step size $\rho>0$ becomes small. We can proceed as before in the case of the simpler scheme of item i) of section 2.

 If on the other hand non-$P$ holds at time $l$ then we have a different control function and get
\begin{equation}\label{growthsecondstep}
\begin{array}{ll}
\delta v^{r,\rho,l+1,ij}(l+1,x)=v^{r,\rho,l+1,ij}(l,x)-v^{r,\rho,l,ij}(l-1,x)\\
\\
=v^{r^{l},\rho,l+1,ij}(l,x)-v^{r,\rho,l,ij}(l-1,x)+\delta r^{l+1,ij}(l+1,x)\\
\\
=\delta v^{r^{l-1},\rho,l+1,ij}(l,x)+\delta r^{l+1,ij}(l+1,x)\\
\\
=\delta v^{r^{l},\rho,l+1,ij}(l,x)+\int_{l}^{l+1}\int_{U_j}\phi^{r,l,ij}(l-1,y)C_{U_j}p^{l+1,ijj}(1,x-y)dyds,
\end{array}
\end{equation}
and
\begin{equation}\label{growthsecondstepr}
\begin{array}{ll}
\delta r^{l+1,ij}(l+1,x)=-\int_{l}^{l+1}\int_{U_j}\phi^{r,l,ij}(l-1,y)C_{U_j}p^{l+1,ijj}(1,x-y)dyds.
\end{array}
\end{equation}
Since non-$P$ holds at $l$ the integrand  $\phi^{r,l,ij}(l-1,y)$ is larger than one for all $y$ where $P$ is violated. However, since $P$ holds at time $l-1$ we have
\begin{equation}
\sup_{x\in U_j}{\big |}r^{l-1,ij}(l-1,x){\big |}\leq C~\mbox{for all }1\leq i\leq n,~j\in J,
\end{equation}
which implies that 
\begin{equation}
\sup_{x\in U_j}{\big |}r^{l,ij}(l,x){\big |}\leq C+1~\mbox{for all }1\leq i\leq n,~j\in J,
\end{equation}
and, according to (\ref{growthsecondstepr}) and for small time step size we get
\begin{equation}
\sup_{x\in U_j}{\big |}r^{l+1,ij}(l+1,x){\big |}\leq C~\mbox{for all }1\leq i\leq n,~j\in J,
\end{equation}
A similar argument holds for multivariate spatial derivatives of order up to $2$ of the control function. 
Hence we get
\begin{equation}\label{preservation1r}
\begin{array}{ll}
\max_{1\leq i\leq n,j\in J}\sum_{0\leq |\alpha|\leq 2}\sup_{x\in U_j}{\big |}D^{\alpha}_xr^{l-1,ij}(l-1,x){\big |}\leq C\\
\\
\rightarrow \max_{1\leq i\leq n,j\in J}\sum_{0\leq |\alpha|\leq 2}\sup_{x U_j}{\big |}D^{\alpha}_xr^{l+1,ij}(\tau,x){\big |}\leq C
\end{array}
\end{equation}
for small time step size even if the property $P$ is violated at time $l$. Furthermore as the property $P$ is satisfied at time $l-1$ we have
\begin{equation}
\sup_{x\in U_j}{\big |}v^{l,\rho,l,ij}(l,x){\big |}\leq C~\mbox{for all }1\leq i\leq n,~j\in J,
\end{equation}
at time $l$ since $P$ is violated at time $l$, i.e., non-$P$ holds at time $l$ we may have no upper bound $C$ at the next time step. This may occur if at some argument $x$ at time $l$ the modulus of the control function becomes larger than $C$, while the modulus of the controlled velocity function is also close to $C$ (at least greater than $C-1$) and at this argument both values have opposite sign (otherwise, if the signs were equal, the control function in the case non-$P$ ensures that the controlled velocity function decreases with the control function from time $l$ to time $l+1$). Well as the property $P$ is assumed to hold at time $l-1$ for small time step size $\rho>0$ we surely have
\begin{equation}
\sup_{x\in U_j}{\big |}v^{l+1,\rho,l,ij}(l+1,x){\big |}\leq C+\frac{1}{2}~\mbox{for all }1\leq i\leq n,~j\in J,
\end{equation}
although the upper bound $C+1$ would suffice for our argument. As the property $P$ holds again at time $l+1$ we have
\begin{equation}
\sup_{x\in U_j}{\big |}v^{l+2,\rho,l,ij}(l+2,x){\big |}\leq C~\mbox{for all }1\leq i\leq n,~j\in J
\end{equation}
by the construction of the control function.
Hence, we have
\begin{equation}\label{preservation1v}
\begin{array}{ll}
\max_{1\leq i\leq n,j\in J}\sum_{0\leq |\alpha|\leq 2}\sup_{x\in  U_j}{\big |}D^{\alpha}_xv^{r,\rho,l-1,ij}(l,x){\big |}\leq C\\
\\
\rightarrow \max_{1\leq i\leq n,j\in J}\sum_{0\leq |\alpha|\leq 2}\sup_{x\in U_j}{\big |}D^{\alpha}_xv^{r,\rho,l+2,ij}(\tau,x){\big |}\leq C,
\end{array}
\end{equation}
and we have established the preservation of the upper bound for the control functions and the controlled velocity functions after mutually two different time steps. It follows that for a constant $C'\leq C+1>0$ independent of the time step number $l$ we have the upper bounds
\begin{equation}
\sup_{l\geq 1}\max_{1\leq i\leq n,j\in J}\sum_{0\leq |\alpha|\leq 2}\sup_{\tau\in [l-1,l],x\in  U_j}{\big |}D^{\alpha}_xv^{r,\rho,l,ij}(\tau,x){\big |}\leq C',
\end{equation}
and 
\begin{equation}
\sup_{l\geq 1}\max_{1\leq i\leq n,j\in J}\sum_{0\leq |\alpha|\leq 2}\sup_{\tau\in [l-1,l],x\in  U_j}{\big |}D^{\alpha}_xvr^{l,ij}(\tau,x){\big |}\leq C'.
\end{equation}
Hence, the scheme is global. It is clear that the estimates can be repeated for higher order derivatives as the local contraction results of step i) and step ii) of this proof hold also for higher order derivatives.

\subsection{step 4: Global existence of classical solutions $\mathbf{v}^{\rho}$ and $\mathbf{v}$}
 Now we have proved that for all $1\leq i\leq n$, all $j\in J$, and for all $l\geq 1$ we have
 \begin{equation}
  r^{l,ij}\in C^{\delta}\left(\left[l-1,l\right]\times U_j  \right),~\mbox{and}~|r^{l,ij}|_{\delta}\leq C
 \end{equation}
and 
\begin{equation}
  v^{r,\rho,l,ij}\in C^{\delta}\left(\left[l-1,l\right]\times U_j  \right) ~\mbox{and}~|v^{r,\rho,l,ij}|_{\delta}\leq C
 \end{equation}
 for a constant $C>0$ which is independent of the time-step number $l$. Indeed we have more regularity with respect to the spatial variables and even with respect to the time-variable $\tau$ (transformed time) we have classical differentiability except at the points $\tau=1,2,\cdots$, i.e., where $\tau$ is a natural number. The non-differentiability in a classical sense with respect to time at these points is due to the fact that the source functions $\phi^{l,ij}$ of the equations for $r^{l,ij}$ are locally constant with respect to the time variable over time $[l-1,l)$ and we have bounded jumps from $\phi^{l-1,ij}$ to $\phi^{l,ij}$ at time $\tau=l-1$ in general. However these source terms appear as a time integral in the representation for $r^{l,ij}$ and for $v^{r,\rho,l,ij}$ (or its first approximation), and this leads to the conclusion that $r^{l,ij}$ and $v^{r,\rho,l,ij}$ are H\"{o}lder continuous across the time points $\tau=l$ for all time step numbers $l\geq 1$. Next, since for both summands in
  \begin{equation}
  v^{\rho,l,ij}=v^{r,\rho,l,ij}-r^{l,ij}.
 \end{equation}
 we have $v^{r,\rho,l,ij},-r^{l,ij}\in C^{\delta}\left(\left[l-1,l\right]\times U_j  \right)$
  we immediately get for all $1\leq i\leq n$, all $j\in J$, and all $l\geq 1$ that 
 \begin{equation}
 v^{\rho,l,ij}\in C^{\delta}\left(\left[l-1,l\right]\times U_j  \right) ~\mbox{and}~|v^{\rho,l,ij}|_{\delta}\leq 2C,
 \end{equation}
 with the same $C>0$ independent of $l\geq 1$.
 Note that for all $l\geq 1$ the local regularity results imply that
$v^{r,\rho,l,ij},-r^{l,ij}\in C^{1,2}\left(\left(l-1,l\right]\times U_j  \right)$.
Hence, for all $1\leq i\leq n$, all $j\in J$, and all $l\geq 1$ w have 
 \begin{equation}
 v^{\rho,l,ij}\in C^{1,2}\left(\left(l-1,l\right]\times U_j  \right) ~\mbox{and}~|v^{\rho,l,ij}|_{\delta}\leq 2C,
 \end{equation}
Hence we have global functions 
\begin{equation}
 v^{\rho,ij}\in C^{\delta}\left(\left[0,\infty\right)\times U_j  \right) ~\mbox{and}~|v^{\rho,ij}|_{\delta}\leq 2C,
\end{equation}
where $v^{\rho,ij}$ is the function which equals $v^{\rho,l,ij}$ if restricted to $[l-1,l)\times U_j$ for all $l\geq 1$ and which solve the incompressible Navier Stokes equation system classically on local domains $[l-1,l)\times U_j$. The next observation from the argument of the preceding 
steps is that the first spatial derivatives of $r^{l,ij}$ and of $v^{r,\rho,l,ij}$ exist continuously for all $l\geq 1$ for all $1\leq i\leq n$ and all $j\in J$. Therefore, we have for all $l\geq 1$, for all $1\leq i\leq n$, all $j\in J$, and all $1\leq k\leq n$
\begin{equation}
 v^{\rho,ij}_{,k}\in 
 C\left(\left[0,\infty\right)\times U_j  \right)
  ~\mbox{and}~|v^{\rho,ijj}|_{C^0\left( \overline{U_j}\right)}\leq 2C.
\end{equation}
This holds for all $j\in J$. Then looking at (\ref{eqiloc1}) we observe that the first order coefficient $v^{kj}$ satisfies 
\begin{equation}
v^{kj}\in 
 C^{\delta}\left(\left[0,\infty\right)\times U_j  \right),
\end{equation}
i.e. the first order coefficients are H\"{o}lder continuous, and this holds also for the intergal terms on the right side (invoking regularity results for Poisson equations). Hence, from classical theory of linear parabolic equations we get for all $1\leq i\leq n$ and all $j\in J$
\begin{equation}
v^{\rho,ij}\in C^{1,2}\left(\left[0,\infty \right)\times U_j  \right).
\end{equation}
It follows immediately that this holds also without time dilatation, i.e., for all $1\leq i\leq n$ and all $j\in J$
\begin{equation}
v^{ij}\in C^{1,2}\left(\left[0,\infty \right)\times U_j  \right),
\end{equation}
where we recall that $v^{ij}(t,.)=v^{\rho,ij}(\tau,.)$ along with $t=\rho\tau$. 
Hence, for the velocity components $v^{ij}$ to $U_j$ we have
\begin{equation}
v^{ij}\in C^{1,2}\left(\left[0,\infty \right)\times U_j  \right)
\end{equation}
for all $1\leq i\leq n$ and all $j\in J$. It follows that
\begin{equation}
\mathbf{v}\in C^{1,2}\left(\left[0,\infty \right)\times M,TM  \right), 
\end{equation}
and our argument is finished.
\end{proof}
 Now assume that
Finally we note that the existence of a classical solution implies that the solution is smooth. Assume that for $\tau\leq l-1$ it has been proved that $\mathbf{v}^{r,\rho,l-1}\in C^{\infty}\left(\left[0,l-1\right] \times M,TM\right)$. We then can extend our proofs of local contraction results in step 1 and step 2 of the proof of the main theorem to function spaces $C^{m,2m}\left(\left[ l-1,l\right]\times U_j \right)$ for any given $m\geq 2$ and repeat the proof with an adapted time size $\rho$ in order to get a global solution $v^{ij}\in C^{m,2m}\left(\left[ 0,\infty\right)\times U_j \right),~1\leq i\leq n,j\in J$. Since this is true for all given $m$ the solution is smooth. Alernatively, as a classical solution is known
we may apply classical regularity theory of linear parabolic equations in order to prove higher regularity. We may just apply a standard theorem of the form
\begin{thm}
Assume that for all multiindices $\alpha$ with $|\alpha|\leq m$ we know that
 \begin{equation}
 D^{\alpha}_xa_{ij},~D^{\alpha}_x b_i,~D^{\alpha}_x c
 \end{equation}
are H\"{o}lder continuous in a domain $D=\left( 0,\infty\right) \times \Omega$ for some domain $\Omega\subset {\mathbb R}^n$. If $u$ is a solution of
\begin{equation}
\frac{\partial u}{\partial t}-\sum_{ij=1}^na_{ij}\frac{\partial^2 u}{\partial x_i\partial x_j}-\sum_{i=1}^nb_i\frac{\partial u}{\partial x_i}-cu=f 
\end{equation}
in $D$, then 
\begin{equation}
D^{\alpha}_xu,~D_tD^{\beta}_xu
\end{equation}
exist for $0\leq |\alpha|\leq m+2$ and $0\leq |\beta|\leq m$ and are all H\"{o}lder continuous in $D$.
\end{thm}
inductively with the order of derivatives. We have   
\begin{cor} For s smooth Riemannian manifold $M$ a viscosity constant $\nu >0$ and data $\mathbf{h}\in C^{\infty}(M,TM)$ we have
\begin{equation}
\mathbf{v}\in C^{\infty}\left(\left[0,\infty\right)\times M,TM\right)
\end{equation}
\end{cor}

The explanation give here for a global scheme of the incompressible Navier-Stokes equation uses strict and uniform parabolic. It is clear that this assumption cannot be removed completely as it seems very likely that solution of Euler-equations (with viscosity $\nu=0$) may blow up in finite time. We shall give an argument for this elsewhere. However, the method considered here may be extended to systems where the second order coefficient functions satisfy a H\"{o}rmander condition. The reason is that in this case Gaussian density estimates of Stroock-Kusuoka type seem to be sufficient.
Each time step then involves the solution of problems on $[0,\infty)\times M$ of the form
\begin{equation}
	\label{projectiveHoermanderSystemgenx2}
	\left\lbrace \begin{array}{ll}
		\frac{\partial u}{\partial t}=\frac{1}{2}\sum_{i=1}^mV_i^2u+V_0u\\
		\\
		u(0,x)=f(x).
	\end{array}\right.
\end{equation}
for $m$ vector fields which look locally like 
\begin{equation}
V_i=\sum_{j=1}^n v_{ji}(x)\frac{\partial}{\partial x_j},
\end{equation}
where $0\leq i\leq m$. For equation on manifolds the H\"{o}rmander condition has to be rephrased in local charts, of course. However, independence of the chart and well-definiteness is straightforward. Denoting the vector fields $V_i$ in a given chart on an open set $\Omega\subset {\mathbb R}^n$ by $V_i$ again (same name) we may say that the H\"{o}rmander condition is satisfied at $x\in \Omega$ if  
\begin{equation}\label{Hoergenx}
\begin{array}{ll}
H_x:=\mbox{span}{\Big\{} &V_i(x), \left[V_j,V_k \right](x), \\
\\
&\left[ \left[V_j,V_k \right], V_l\right](x),\cdots |1\leq i\leq m,~0\leq j,k,l\cdots \leq m {\Big \}}
\end{array}
\end{equation}
holds. We may say that the H\"{o}rmander condition is satisfied on $\Omega$ if (\ref{Hoergenx}) is satisfied for all $x\in \Omega$.
Note that second order equations of type (\ref{projectiveHoermanderSystemgenx2}) correspond to diffusion processes $X$ which have a local representation on domains $\Omega$ in components, and satisfy stochastic ODEs. In the framework of Malliavin calculus it was proved that 
\begin{thm}
Let the assumption of (\ref{Hoergenx}) be satisfied for all $x\in \Omega$ and let $T>0$. 
Then the law of the diffusion process $X$ (corresponding to the second order equation (\ref{projectiveHoermanderSystemgenx2}) in the Feynman-Kac sense) exists on a domain $\Omega\subseteq {\mathbb R}^n$ is absolutely continuous with respect to the Lebesgue measure, and the density $p$ exists and is smooth, i.e., on a domain $\Omega\subseteq {\mathbb R}^n$ we have
\begin{equation}
\begin{array}{ll}
p:(0,T]\times \Omega\times \Omega\rightarrow {\mathbb R}\in C^{\infty}\left( (0,T]\times \Omega\times \Omega\right). 
\end{array}
\end{equation}
Moreover, for each nonnegative natural number $j$, and multiindices $\alpha,\beta$ there are increasing functions of time
\begin{equation}\label{constAB}
A_{j,\alpha,\beta}, B_{j,\alpha,\beta}:[0,T]\rightarrow {\mathbb R},
\end{equation}
and functions
\begin{equation}\label{constmn}
n_{j,\alpha,\beta}, 
m_{j,\alpha,\beta}:
{\mathbb N}\times {\mathbb N}^d\times {\mathbb N}^d\rightarrow {\mathbb N},
\end{equation}
such that 
\begin{equation}\label{pxest}
{\Bigg |}\frac{\partial^j}{\partial t^j} \frac{\partial^{|\alpha|}}{\partial x^{\alpha}} \frac{\partial^{|\beta|}}{\partial y^{\beta}}p(t,x,y){\Bigg |}\leq \frac{A_{j,\alpha,\beta}(t)(1+x)^{m_{j,\alpha,\beta}}}{t^{n_{j,\alpha,\beta}}}\exp\left(-B_{j,\alpha,\beta}(t)\frac{(x-y)^2}{t}\right) 
\end{equation}
Moreover, all functions (\ref{constAB}) and  (\ref{constmn}) depend on the level of iteration of Lie-bracket iteration at which the H\"{o}rmander condition becomes true.
\end{thm} 
These density estimates fit in our scheme and may lead to generalisations. Note that polynomial growth factor  $(1+x)^{m_{j,\alpha,\beta}}$ in (\ref{pxest}) is no obstacle since we work on compact manifolds. We considered a natural class of Navier Stokes equations systems in \cite{KHNS}, where the H\"{o}rmander condition is satisfied for the uncoupled second order diffusion part of the operator which is identical for all velocity components. The proof simplifies in the case of compact manifolds, since we do not have to deal with the polynomial decay at infinity and the additional complications related to the additional polynomial growth factor $(1+x)^{m_{j,\alpha,\beta}}$ in the estimate (\ref{pxest}).
  
\footnotetext[1]{
\texttt{{kampen@wias-berlin.de}, {kampen@mathalgorithm.de}}.}

\end{document}